\documentclass[10pt,oneside]{article}

\usepackage[utf8]{inputenc} %
\usepackage[T1]{fontenc}    %
\usepackage{url}            %
\usepackage{booktabs}       %
\usepackage{amsfonts}       %
\usepackage{nicefrac}       %
\usepackage{microtype}      %

\usepackage{lmodern}
\usepackage{times}

\usepackage{amssymb,amsmath,amsthm}
\usepackage{bbold}
\usepackage[short]{optidef}

\usepackage{framed,multirow,multicol}

\usepackage{xcolor,xspace}
\usepackage{lscape}
\usepackage{subfigure,graphicx,epsfig,tikz,caption}
\usepackage[normalem]{ulem}
\usepackage{enumerate}
\usepackage{verbatim}
\usepackage{makeidx,latexsym}
\usepackage[colorlinks=true,citecolor=blue]{hyperref}%

\usepackage{bm}
\usepackage{stmaryrd}
\usepackage{lscape}

\usepackage[numbers,sort&compress,square,comma]{natbib}

\usepackage{algorithmic,algorithm}

\usepackage[english]{babel}
\usepackage{bbm}
\usepackage{pgfplots}

\usepackage{parskip}
\usepackage[letterpaper,margin=1.1in]{geometry}

\usepackage{thmtools,thm-restate}

\declaretheorem{lemma}

\declaretheoremstyle[qed=$\square$]{definitionwithend}
\declaretheorem[style=definitionwithend]{definition}

\PassOptionsToPackage{numbers, compress}{natbib}

\DeclareMathOperator{\conv}{conv}

\DeclareMathOperator{\cone}{cone}

\newcommand{\M}{\mathbb{M}}

\newcommand{\R}{\mathbb{R}}

\newcommand{\Z}{\mathbb{Z}}

\newcommand{\epr}{\hfill\hbox{\hskip 4pt \vrule width 5pt height 6pt depth 1.5pt}\vspace{0.0cm}\par}

\newtheorem{theorem}{Theorem}[section]

\newtheorem*{proof-main theorem}{Proof of Theorem~\ref{main result}}
\newtheorem*{proof-lemma}{Proof of Lemma~\ref{lemma:double-star}}

\newcounter{claim_nb}[theorem]
\newtheorem{claim'}[claim_nb]{Claim}
\newtheorem*{claim*}{Claim}

\theoremstyle{definition}

\newcommand{\floor}[1]{\lfloor #1 \rfloor}
\newcommand{\ceil}[1]{\lceil #1 \rceil}
\newcommand{\scg}{$S$-CG cut~}
\newcommand{\scgx}{$S$-CG cut}
\newcommand{\scgs}{$S$-CG cuts~}

\newcommand{\scgc}{$S$-CG closure~}

\newcommand{\supp}{\text{supp}}

\DeclareMathOperator{\rec}{rec}
\DeclareMathOperator{\lin}{lin}
\DeclareMathOperator{\proj}{proj}

\newcommand {\beqn}{\begin{equation}}\newcommand {\eeqn}{\end{equation}}
\newcommand {\beqan}{\begin{eqnarray}}\newcommand {\eeqan}{\end{eqnarray}}
\newcommand {\beqa}{\begin{eqnarray*}}\newcommand {\eeqa}{\end{eqnarray*}}
\newcommand{\skipit}  [1] {}

\begin{document}
	\title{On a generalization of the Chv\'atal-Gomory closure}
	\author{
Sanjeeb Dash\thanks{IBM Research, Yorktown Heights, NY 10598, USA, \url{sanjeebd@us.ibm.com}}\and Oktay G\"{u}nl\"{u}k\thanks{School of ORIE, Cornell University, Ithaca, NY 14850, USA, \url{ong5@cornell.edu}} 	
\and
Dabeen Lee\thanks{Discrete Mathematics Group, Institute for Basic Science (IBS), Daejeon 34126, Republic of Korea, \url{dabeenl@ibs.re.kr}}	}	
	\maketitle
	
\begin{abstract}
Many practical integer programming problems involve variables with one or two-sided bounds. 
Dunkel and Schulz (2012) considered a strengthened version of Chv\'{a}tal-Gomory (CG) inequalities that use 0-1 bounds on variables, and showed that the set of points in a rational polytope that satisfy all these strengthened inequalities is a polytope. 
Recently, we generalized this result by considering strengthened CG inequalities
that use all variable bounds. In this paper, we generalize further by considering not just variable bounds, but general linear constraints on variables. We show that all points in a rational polyhedron that satisfy such strengthened CG inequalities form a rational polyhedron. We also extend this polyhedrality result to mixed-integer sets defined by linear constraints. 
\end{abstract}

\section{Introduction}

Gomory~\cite{G1958} discovered the first finitely convergent cutting plane algorithm -- based on Gomory fractional cuts -- for solving integer linear programs.
Chv\'atal~\cite{C1973} later studied a related cut-generation scheme,
where the generated cuts are called \emph{Chv\'atal-Gomory (CG) cuts}, and are essentially equivalent to Gomory fractional cuts.
CG cuts are prevalent in the discrete optimization literature. Many fundamental classes of facet-defining inequalities for combinatorial optimization problems are CG cuts, e.g., odd set inequalities for the matching problem~\cite{matching,C1973} and odd circuit inequalities for the stable set problem~\cite{GS1986}. CG cuts are computationally effective for solving integer linear programs in practice~\cite{first-closure,projected-cg}, and CG cuts for nonlinear integer programs have also been studied~\cite{conic-cg}. Some important classes of inequalities used for binary polynomial optimization are CG cuts~\cite{polynomial-cg}.

An important property of CG cuts proved by Schrijver~\cite{S1980} %
is that although there are infinitely many CG cuts for a given rational polyhedron, the list of \emph{nonredundant} CG cuts is always finite. Equivalently, the \emph{Chv\'atal-Gomory (CG) closure} of a rational polyhedron, defined as the set of points satisfying all possible CG cuts, is again a rational polyhedron. A number of recent papers prove the polyhedrality of the CG closure for more general closed convex sets such as irrational polytopes~\cite{DS2013}, rational ellipsoids~\cite{DV2010}, strictly convex sets~\cite{DDV2011}, and finally arbitrary compact convex sets~\cite{DDV2014,BP2014} (unlike bounded convex sets, a polyhedron with an irrational ray may have infinitely many nonredundant CG cuts).

In this paper, we take a different direction of generalizing Schrijver's polyhedrality result. We consider a strengthening of CG cuts for a rational polyhedron that we explain below. Given a rational polyhedron $P\subseteq\R^n$ and a valid inequality $\alpha x \leq \beta$ with integer coefficients $\alpha\in\mathbb{Z}^n$, the CG cut derived from $\alpha x \leq \beta$ is defined as $\alpha x \leq \floor{\beta}$. Note that 
\[ \floor{\beta} \geq \max\{\alpha x :\ x \in \Z^n,\ \alpha x \leq \beta \}\]
and the inequality becomes an equality if the coefficients of $\alpha$ are coprime integers. The gap between $\floor{\beta}$ and $\max\{\alpha x :x\in P\cap \Z^n \}$ can be large, and $\alpha x \leq \beta'$ can be a valid inequality for $P \cap \Z^n$ for some $\beta'$ that is much smaller than $\beta$.
If we are given {\em a priori} information that $P\cap \Z^n$ is contained in some $S\subseteq \Z^n$, then assuming $S$ has a point satisfying $\alpha x \leq \beta$, the inequality $\alpha x \leq \floor{\beta}_{S,\alpha}$,
where
\[ \floor{\beta}_{S,\alpha} = \max\{\alpha x :\ x \in S,\ \alpha x \leq \beta\}\]
is certainly valid for $P\cap \Z^n$ and is a strengthening of $\alpha x \leq \floor{\beta}$. We call the inequality $\alpha x \leq \floor{\beta}_{S,\alpha}$ an \emph{$S$-Chv\'{a}tal-Gomory ($S$-CG) cut} for $P$. If $S$ does not contain a point satisfying $\alpha x \leq \beta$, then $P \cap S$ is empty, in which case, we say that $\mathbf{0}x \leq -1$ is an \scg for $P$. For general $S$, the hyperplane $\left\{x\in\mathbb{R}^n:\alpha x=\beta\right\}$ is moved until it hits a point in $S$ (see Figure~\ref{fig:scgcut}); the resulting hyperplane is given by $\left\{x\in\mathbb{R}^n:\alpha x=\floor{\beta}_{S,\alpha}\right\}$. 
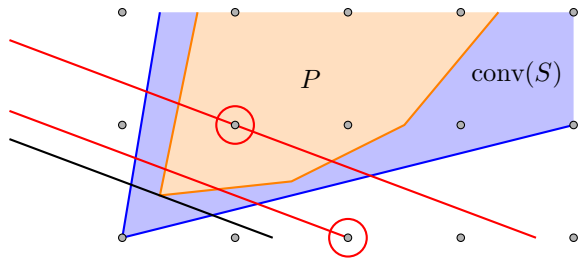
\begin{figure}[h!]
	\begin{center}
		\begin{tikzpicture}
		[main node/.style={circle,fill=black!30,draw,minimum size=0.1em, inner sep=1pt}]
		
		\fill[blue, nearly transparent] (4.5,4) -- (4,1) -- (10,2.5) -- (10,4) -- (9,4) -- (7.75,2.5) -- (6.25,1.75) -- (4.5,1.5625) -- (5,4) -- (4.5,4);
		\draw[blue,thick] (4.5,4) -- (4,1) -- (10,2.5);
		\draw[orange,thick]  (5,4) -- (4.5,1.5625) -- (6.25,1.75) -- (7.75,2.5) -- (9,4);
		\fill[orange,nearly transparent]  (5,4) -- (4.5,1.5625) -- (6.25,1.75) -- (7.75,2.5) -- (9,4) -- (5,4);
		\draw[black,thick] (2.5,2.3125) -- (6,1);
		\draw[red,thick] (7,1) -- (2.5,2.6875);
		\draw[red,thick] (2.5,3.63) -- (5.5,2.5) -- (9.5,1);
		\draw[red,thick] (7,1) circle (0.25cm);
		\draw[red,thick] (5.5,2.5) circle (0.25cm);
		
		\node[main node] (4) at (5.5,1) {};
		\node[main node] (5) at (4,1) {};
		\node[main node] (6) at (7,1) {};
		\node[main node] (7) at (8.5,1) {};
		\node[main node] (8) at (10,1) {};
		
		\node[main node] (12) at (4,2.5) {};
		\node[main node] (13) at (5.5,2.5) {};
		\node[main node] (14) at (7,2.5) {};
		\node[-,label={$P$}] (14) at (6.5,2.75) {};
		\node[-,label={$\conv(S)$}] (14) at (9.25,2.75) {};
		\node[main node] (15) at (8.5,2.5) {};
		\node[main node] (16) at (10,2.5) {};
		
		\node[main node] (20) at (4,4) {};
		\node[main node] (21) at (5.5,4) {};
		\node[main node] (22) at (7,4) {};
		\node[main node] (23) at (8.5,4) {};
		\node[main node] (24) at (10,4) {};
		
		\end{tikzpicture}
		\caption{Comparison of the CG cut and the $S$-CG cut from an inequality}\label{fig:scgcut}
	\end{center}
\end{figure}
In a similar manner, we define
\[ \ceil{\beta}_{S,\alpha} = \min\{\alpha x :\ x \in S,\ \alpha x \geq \beta\},\]
assuming that $S$ has a point satisfying $\alpha x \geq \beta$.
Then we say that $\alpha x \geq \ceil{\beta}_{S,\alpha}$ is the \scg obtained from $\alpha x \geq \beta$. This way of strengthening CG cuts was considered earlier in~\cite{DS2012,P2011}. Based on this generalization of CG cuts, we can also extend the notion of closure. We define the \emph{$S$-Chv\'atal-Gomory ($S$-CG) closure} of a polyhedron to be the set of all points that satisfy all \scgs for the polyhedron. For the case $S = \Z^n$, the \scgs are essentially equivalent to the CG cuts, and the \scgc coincides with the CG closure. 

A natural question is whether the \scgc of a rational polyhedron is also a rational polyhedron. Dunkel and Schulz~\cite{DS2012} proved that when $S = \{0,1\}^n$ and $P$ is a rational polytope contained in $[0,1]^n$, then the $S$-CG closure of $P$ is also a rational polytope. In~\cite{Dash19}, we observed that a modification of their argument works for any finite $S$. To be precise, we proved that when $S$ is finite, the $S$-CG closure of a rational polyhedron $P$ is a rational polyhedron, regardless of whether or not $P\subseteq \conv(S)$. Furthermore, using a novel proof technique, we showed in \cite{Dash19} that when $S$ is the set of integral points that satisfy an arbitrary collection of variable bounds and $P$ is a rational polyhedron contained in $\conv(S)$, the $S$-CG closure of $P$ is a rational polyhedron. This result covers the cases $S=\{0,1\}^n$, $S=\Z^n$, and $S=\Z_+^n$. 

In this paper, we consider the case when $S$ is the set of integer points in an arbitrary rational polyhedron. The following theorem is our main result:
\begin{theorem}\label{main result}
Let $S=R\cap \mathbb{Z}^n$ for some rational polyhedron $R$ and $P\subseteq\conv(S)$ be a rational polyhedron. Then the $S$-CG closure of $P$ is a rational polyhedron.
\end{theorem}
We next give a high-level outline of some of the technical details of the proof. In Section~\ref{sec:cylinder-Chvatal}, we start by proving that the result holds when $R$ is a rational cylinder. The key idea there is to use a unimodular mapping of $R$ to a set of the form $T\times \R^{l}$ where $l \leq n$ and $T \subseteq \R^{n-l}$ is a polytope. The case $R = T \times \R^{l}$ is already covered in~\cite{Dash19}. In Section~\ref{sec:pointed}, we then consider the case when $R$ is a pointed polyhedron. The hardest case in~\cite{Dash19} is the case when $S = \Z^n_+$ and $P$ is a packing or covering polyhedron contained in $\R^n_+$. Similarly, the case when $R$ is a pointed polyhedron and $P$ behaves like a packing or covering polyhedron with respect to $R$ is the hardest case in this paper. The main technical difference between $R$ being a pointed polyhedron and $R$ being $\conv(\mathbb{Z}_+^n)=\mathbb{R}_+^n$ is that $R$ can have more than $n$ extreme rays and, in particular, the extreme rays can be linearly dependent. Nevertheless, this case can be dealt with by generalizing the argument in~\cite{Dash19} to our setting. We essentially prove that given a valid inequality for $P$ (and the associated hyperplane) that yields a nonredundant \scgx, the points at which the hyperplane intersects the rays of the recession cone of $R$ are bounded. %
In Section~\ref{sec:main}, we consider the case when $R$ is a polyhedron with a nontrivial lineality space, completing the proof of Theorem~\ref{main result}. 

In Section~\ref{sec:mixed}, we extend our main result to the mixed-integer setting. Bonami et al.~\cite{projected-cg} defined \emph{projected Chv\'atal-Gomory cuts} as a generalization of CG cuts for mixed-integer linear programs. We generalize projected CG cuts and define $S$-CG cuts for polyhedral mixed-integer sets by defining $S$ to be an appropriate mixed-integer set and defining the $S$-CG closure accordingly. Given a rational polyhedron $P \subseteq \R^{n}\times\R^l$ and the requirement that the first $n$ variables are integral, we only consider valid inequalities for $P$ that have nonzero components only for the integer variables, and define $S$ to be $R\cap (\Z^n\times\R^l)$ where $R$ is a rational polyhedron in $\R^{n}\times\R^l$. We conclude in Section~\ref{sec:concl} with some remarks on possible generalizations of our results.

\subsection{Related work}

\emph{Split cuts}~\cite{disjunctive,split-closure} form an important class of {\em intersection cuts}, introduced by Balas \cite{int3},  %
and are obtained from \emph{splits}. CG cuts are a special case of split cuts, as a CG cut is obtained from a \emph{split disjunction} that has one of its sides empty. Recently, intersection cuts from general \emph{lattice-free sets} and \emph{$S$-free sets} have attracted enormous attention from the optimization community~\cite{int0,int1,int2,int4,int51,int52,int6,int7,int8,int10,int11,int111,int13,int14,int141,int15,int16,BLTW2017}. Just as CG cuts form a special class of split cuts, $S$-CG cuts can be interpreted as intersection cuts from \emph{$S$-free splits}, and equivalently, \emph{wide splits} -- a name coined by Bonami et al.~\cite{BLTW2017}.

Several families of lattice-free sets and the associated cuts and closures have been introduced, and the corresponding polyhedrality theorems for rational polyhedra were proved~\cite{max-facet-width,split-closure,split-closure-2,split-closure-3,averkov,int10,int51,int2}. Hence, it is natural to ask if standard techniques from these papers as well as the papers on the CG closure~\cite{S1980,DS2013,DV2010,DDV2011,DDV2014,BP2014} can be applied for proving the polyhedrality of the $S$-CG closure. However, the earlier results rely directly or indirectly on the assumption that certain lattice-free sets have bounded \emph{max-facet-width} (see~\cite{int2}), which is defined as follows. The \emph{width} of a convex set $L$ along a vector $\pi$ is defined to be the number $w(L,\pi):=\max\{\pi x:x\in L\}-\min\{\pi x:x\in L\}$. Given a rational polyhedron $L$ whose facets are defined by inequalities $\pi^ix\geq \pi_0^i$ for $i=1,\ldots,k$ where $\pi^i$ have coprime integer coefficients, the max-facet-width of $L$ is defined as $\max\{w(L,\pi^i):i=1,\ldots,k\}$.
Recall that we obtain the $S$-CG cut $\alpha x\leq \floor{\beta}_{S,\alpha}$ from a valid inequality $\alpha x\leq \beta$. Here, the gap $\beta-\floor{\beta}_{S,\alpha}$ can grow as a function of the components in $\alpha$ when $S\neq\mathbb{Z}^n$, implying in turn that $S$-free splits do not necessarily have bounded max-facet-width.

There are other closure operations related to our work. The $S$-CG cuts for the case $S=\{0,1\}^n$ are valid for the 0-1 knapsack set $\{x \in \{0,1\}^n : \alpha x \leq \beta\}$; valid inequalities for such knapsack sets were used to solve practical problem instances in Crowder et al.~\cite{CJP83}, and an associated closure operation was defined by Fischetti and Lodi~\cite{FS10}. Fukasawa and Goycoolea~\cite{fukasawa} studied valid inequalities for bounded and unbounded knapsack sets of the form $\{x \in \mathbb{Z}^n : \ell\leq x\leq u,\ \alpha x \leq \beta\}$ where $\ell\in (\mathbb{R}\cup\{-\infty\})^n$ and $u\in (\mathbb{R}\cup\{+\infty\})^n$, for which $S$-CG cuts with $S=\{x \in \mathbb{Z}^n : \ell\leq x\leq u\}$ are valid. Bodur et al.~\cite{bodur} introduced the notion of \emph{aggregation closure} which is defined as the set of points satisfying valid inequalities for all knapsack sets $\{x \in \Z_+^n: \alpha x \leq \beta\}$ where $\alpha x \leq \beta$ is valid for $P$ and $\alpha \leq 0$ or $\alpha \geq 0$. Pashkovich et al. \cite{PPP19} showed that the aggregation closure  is polyhedral for packing and covering polyhedra.
For packing polyhedra, Del Pia et al.~\cite{dlz} independently proved the same result.

\subsection{Formal definition of the $S$-CG closure}\label{sec:prelim}

Given a rational polyhedron $P = \{x\in \R^n : Ax  \leq b\}$ where  $A\in\Z^{m\times n}$ and $b\in\Z^m$, we define $\Pi_P$ as the set of all coefficient vectors that define valid, supporting inequalities for $P$ with integral left-hand-side coefficients:
\begin{equation}\label{eq:pip}
\Pi_P = \left\{ (\lambda A, \lambda b) \in \Z^n\times\R:\;\lambda\in\mathbb{R}_+^m,\;{\lambda b}= \max\{\lambda Ax : x \in P\}\right\}.
\end{equation}
Hence, for $(\alpha,\beta)\in \Pi_P$, $\alpha x\leq \beta$ is an inequality that is supporting and valid for $P$.
Finally, for $\Omega\subseteq\Pi_P$, we define $P_{S,\Omega}$ as $\bigcap_{(\alpha,\beta)\in \Omega}\left\{x\in\mathbb{R}^n:\;\alpha x\leq \floor{\beta}_{S,\alpha}\right\}$. Then the \scgc of $P$ can be formally defined as $P_{S,\Pi_P}$. Throughout the paper, we denote by $P_S$ the \scgc of $P$:
\begin{equation}\label{eq:scg}
P_{S}:= P_{S,\Pi_P}=\bigcap\limits_{(\alpha,\beta)\in \Pi_P}\left\{x\in\mathbb{R}^n:\;\alpha x\leq \floor{\beta}_{S,\alpha}\right\}.	
\end{equation}
Notice that if $\Gamma\subseteq \Omega \subseteq \Pi_P$, then $P_{S} \subseteq P_{S,\Omega}\subseteq P_{S,\Gamma}$. Also, if $S\subseteq T$ for some $T\subseteq\mathbb{Z}^n$, then $P_S\subseteq P_{T}$. Likewise, for any $\Omega\subseteq\Pi_P$, we have $P_{S,\Omega}\subseteq P_{T,\Omega}$ if $S\subseteq T$.

Throughout the paper, we assume that $P$ and $S$ are nonempty. If $P$ is empty, Farkas' lemma (see~\cite[Theorem 3.4]{ipbook}) implies that ${\bf 0}x\leq -1$ can be derived from $Ax\leq b$, in which case, $P_S$ is trivially empty. If $S$ is empty, then the assumption that $P\subseteq \conv(S)$ enforces $P$ empty, and again, $P_S$ is empty.

As $S$ is nonempty and $S$ is the set of integer points contained in a rational polyhedron $R$, it follows from Meyer's theorem~\cite{meyer} that $\conv(S)$ is also a rational polyhedron and the recession cones of $\conv(S)$ and $R$ coincide (see also~\cite[Theorem 4.30]{ipbook}). 

We assume basic knowledge of polyhedral theory in relation to integer programming. For basic terminologies and definitions, we refer the reader to a comprehensive textbook in the area~\cite{ipbook}.	
\section{Integer points in a general cylinder}\label{sec:cylinder-Chvatal}

We say that a rational polyhedron $R$ is a {\it rational cylinder} if the recession cone and lineality space of $R$ are the same. In this section, we consider the case when $S=R\cap\mathbb{Z}^n$ for some rational cylinder $R$. Note that a rational affine subspace is a rational cylinder but the converse is not always true. For example, the convex hull of $F\times \mathbb{Z}^l$ for any finite $F\subseteq\mathbb{Z}^{n-l}$ is a rational cylinder. For this special case, we already have the following polyhedrality result:
\begin{theorem}[{\bf \cite[Theorem 3.4]{Dash19}}]\label{Thm:ortho-cylinder}
	Let $S=F\times \mathbb{Z}^{l}$ for some finite $F\subseteq\mathbb{Z}^{n-l}$ where $0<l\leq n$. If $P\subseteq\mathbb{R}^{n}$ is a rational polyhedron then $P_S$ is a rational polyhedron.
\end{theorem}
We will extend this result to general rational cylinders by taking appropriate  \emph{unimodular transformations}. Remember that a unimodular transformation is a mapping $\tau:\mathbb{R}^n\rightarrow \mathbb{R}^n$ that maps $x\in\mathbb{R}^n$ to $Ux+v\in\mathbb{R}^n$ for some unimodular matrix $U\in\mathbb{Z}^{n\times n}$ and some integral vector $v\in\mathbb{Z}^n$. 
Note that the inverse mapping $\tau^{-1}(x)=U^{-1}x-U^{-1}v$ is also a unimodular transformation. For $X\subseteq \mathbb{R}^n$, we denote by $\tau(X)$ the image of $X$ under $\tau$. For $\Pi\subseteq\Pi_P$, although $\Pi$ is not in the space of $\mathbb{R}^n$, we abuse our notation and define $\tau(\Pi)$ as $\{(\alpha U^{-1},\beta+\alpha U^{-1}v):(\alpha,\beta)\in \Pi\}\subseteq\Pi_{\tau(P)}$.
\begin{lemma}[\bf Unimodular mapping lemma~\cite{Dash19}]\label{LE:unimodular-closure}
	Let $S\subseteq\mathbb{Z}^n$ and $P\subseteq\conv(S)$ be a rational polyhedron. Then $\tau(P)\subseteq \conv(\tau(S))$, and for any $\Pi\subseteq\Pi_P$, $\tau(P_{S,\Pi})=\tau(P)_{\tau(S),\tau(\Pi)}$. In particular, $\tau(P_S)=\tau(P)_{\tau(S)}$.
\end{lemma}

Essentially, we will argue that the set of integer points in a rational cylinder can be mapped to a set of the form $F\times\mathbb{R}^l$ where $F\subseteq \mathbb{Z}^{n-l}$ is finite by a unimodular transformation.
\begin{theorem}\label{thm:cylinder}
	Let $S=R\cap\mathbb{Z}^n$ for some rational cylinder $R$. If $P\subseteq \conv(S)$ is a rational polyhedron, then $P_S$ is a rational polyhedron.
\end{theorem}
\begin{proof}
	Since $S=R\cap\mathbb{Z}^n$ and $R$ is a rational cylinder, we have $\conv(S)\cap\mathbb{Z}^n=S$ and thus $\conv(S)$ itself is a rational cylinder. Then there exist some integer vectors $v^1,\ldots,v^{g}$ such that $\conv(S)=\conv\left\{v^1,\ldots,v^{g}\right\}+\mathcal{L}$ where $\mathcal{L}$ is the lineality space of $\conv(S)$. Since $\mathcal{L}$ is a linear subspace in $\mathbb{R}^n$ defined by rational data, there exists a unimodular transformation $\tau$ mapping $\mathcal{L}$ to $\{{\bf 0}\}\times\mathbb{R}^{l}$ where $0\leq l\leq n$ is the dimension of $\mathcal{L}$. Then 
	\begin{equation}\label{eq:sec2}
	\tau(\conv(S))=\tau\left(\conv\left\{v^1,\ldots,v^{g}\right\}\right)+\{{\bf 0}\}\times\mathbb{R}^{l}.
	\end{equation}
	 Note that the right-hand side of~\eqref{eq:sec2} equals $\conv\left\{\tau(v^1),\ldots,\tau(v^g)\right\}+\{{\bf 0}\}\times\mathbb{R}^{l}$ and can be written in the form of $\conv(F)\times\mathbb{R}^l$ for some finite $F\subseteq\Z^{n-l}$. As $\tau(\conv(S))=\conv(\tau(S))$ in the left-hand side of~\eqref{eq:sec2}, it follows that $\conv(\tau(S))=\conv(F)\times\mathbb{Z}^l$, which implies that $\tau(S)=F\times\mathbb{Z}^{l}$. Then, by Theorem~\ref{Thm:ortho-cylinder}, $\tau(P)_{\tau(S)}$ is a rational polyhedron, so it follows from Lemma~\ref{LE:unimodular-closure} that $P_S$ is a rational polyhedron, as required. 
	 \ifx\flagJournal\true \qed \fi
\end{proof}	
\section{Integer points in a pointed polyhedron}\label{sec:pointed}

\newcommand {\pc}{{P^{\uparrow}}}
\newcommand {\pp}{{P^{\downarrow}}}

In this section, we consider the case when
\[
S=R\cap\mathbb{Z}^n\quad \text{where }R\text{ is a rational pointed polyhedron}.
\]
Then $\conv(S)\cap \mathbb{Z}^n=S$ and $\conv(S)$ is also a rational pointed polyhedron. We will show that the \scgc of a rational polyhedron $P\subseteq\conv(S)$ is again a rational polyhedron. To simplify the proof, we will reduce this setting to a more restricted setting with additional assumptions on $S$ and $P$, and we will see that these assumptions make the structure of $S$ and that of $P$ easier to deal with. The first part of Section~\ref{sec:pointed} explains the reduction, and Sections~\ref{sec:cov} and~\ref{sec:pack} consider the narrower case of $S$ and $P$ obtained after the reduction.

The first assumption we make is on the structure of $S$. As $\conv(S)$ is a rational polyhedron, $\conv(S)$ can be expressed as the Minkowski sum of the convex hull of integral extreme points and the conic hull of integral extreme rays. Hence, for some integers $g,h\geq 0$, there exist integer vectors $v^1,\ldots,v^{g},r^1,\ldots,r^{h}\in\mathbb{Z}^n$ such that
$\conv(S)$ can be rewritten as
\begin{equation}\label{eq:S_pointed}
\conv(S)=\conv\left\{v^1,\ldots,v^{g}\right\}+\cone\left\{r^1,\ldots,r^{h}\right\}.
\end{equation}
Since $\conv(S)$ is pointed, $\cone\left\{r^1,\ldots,r^{h}\right\}$ has to be pointed as well. Given that $\conv(S)$ has the form of~\eqref{eq:S_pointed}, we assume the following:
\begin{equation}\label{eq:S_pointed_assumption}
\cone\left\{r^1,\ldots,r^{h}\right\}\subseteq \left\{{\bf 0}\right\}\times\mathbb{R}^{n_2},\ \conv(S)\subseteq \cone\left\{e^1,\ldots,e^{n_1},r^1,\ldots,r^{h}\right\}
\end{equation}
where $n_2$ is the dimension of $\rec\left(\conv(S)\right)=\cone\left\{r^1,\ldots,r^{h}\right\}$, $n_1=n-n_2$, and $e^1,\ldots,e^{n_1}$ are unit vectors in $\mathbb{R}^{n_1}\times\{\mathbf{0}\}$. For ease of notation, we use the following notation throughout the paper:
$$\mathcal{C}= \cone\left\{e^1,\ldots,e^{n_1},r^1,\ldots,r^{h}\right\}.$$
Basically, we take a full-dimensional pointed cone containing $\conv(S)$. The assumption~\eqref{eq:S_pointed_assumption} can be justified by taking a unimodular transformation that maps a general $S$ to a set satisfying~\eqref{eq:S_pointed_assumption}.
\begin{figure}[h!]
	\begin{center}
		\begin{tikzpicture}
		[main node/.style={circle,fill=black!30,draw,minimum size=0.1em, inner sep=1pt}]
		
	\fill[blue, nearly transparent] (8,3) -- (7,2) -- (5,1) -- (5,2) -- (6,3)  -- (8,3);
\draw[blue,thick] (8,3) -- (7,2) -- (5,1) -- (5,2) -- (6,3);
		
		\fill[red, nearly transparent] (12,3) -- (12,2) -- (13,1) -- (14,2) -- (14,3) -- (12,3);
		\draw[red,thick] (12,3) -- (12,2) -- (13,1) -- (14,2) -- (14,3);
		
		\draw[black,thick,->] (9,2.5) -- (10,2.5);

		\node[main node,label=below:{$(-5,0)$}] (4) at (5,1) {};
		\node[main node] (5) at (4,1) {};
		\node[main node] (6) at (6,1) {};
		\node[main node] (7) at (7,1) {};
		\node[main node] (8) at (8,1) {};
		
		\node[main node] (12) at (4,2) {};
		\node[main node] (13) at (5,2) {};
		\node[main node] (14) at (6,2) {};
		\node[main node] (15) at (7,2) {};
		\node[main node] (16) at (8,2) {};
		
		\node[main node] (20) at (4,3) {};
		\node[main node] (21) at (5,3) {};
		\node[main node] (22) at (6,3) {};
		\node[main node] (23) at (7,3) {};
		\node[main node] (24) at (8,3) {};

		\node[main node] (4) at (12,1) {};
		\node[main node] (5) at (11,1) {};
		\node[main node,label=below:{$(1,0)$}] (6) at (13,1) {};
		\node[main node] (7) at (14,1) {};
		\node[main node] (8) at (15,1) {};
		
		\node[main node] (12) at (11,2) {};
		\node[main node] (13) at (12,2) {};
		\node[main node] (14) at (13,2) {};
		\node[main node] (15) at (14,2) {};
		\node[main node] (16) at (15,2) {};
		
		\node[main node] (20) at (11,3) {};
		\node[main node] (21) at (12,3) {};
		\node[main node] (22) at (13,3) {};
		\node[main node] (23) at (14,3) {};
		\node[main node] (24) at (15,3) {};
		
		\end{tikzpicture}
		\caption{Obtaining a perpendicular recession cone in $\R^2$}\label{fig:pointed}
	\end{center}
\end{figure}
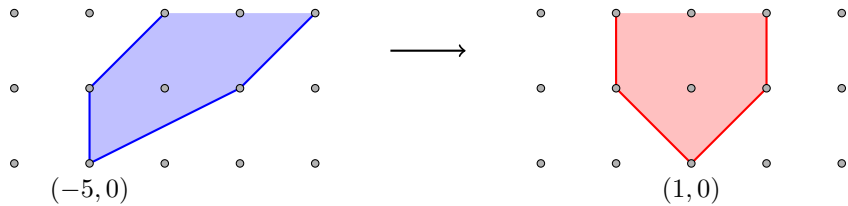
 For example, if $\conv(S)$ is given as the left polyhedron of Figure~\ref{fig:pointed}, $\conv(S)$ does not satisfy~\eqref{eq:S_pointed_assumption} since its ray $(1,1)$ is not contained in $\{0\}\times\mathbb{R}$. Nevertheless, it can be mapped to the polyhedron on the right, by $\tau:(x_1,x_2)\to (6+x_1-x_2,x_2)$, which satisfies~\eqref{eq:S_pointed_assumption}. The following lemma formalizes this observation.
\begin{lemma}\label{reduction1}
Let $S=R\cap\mathbb{Z}^n$ for some rational pointed polyhedron $R$. Then there is a unimodular transformation $\tau$ so that $T:=\tau(S)$ has the property that $\conv(T)\cap\mathbb{Z}^n=T$ and $\conv(T)$ is of the form~\eqref{eq:S_pointed} satisfying~\eqref{eq:S_pointed_assumption}.
\end{lemma}
\begin{proof}
As $R$ is a pointed polyhedron, $\conv(S)$ is also pointed and $\conv(S)\cap\mathbb{Z}^n=S$. 
Let $n_2$ denote the dimension of $\rec\left(\conv(S)\right)$. Since $\rec\left(\conv(S)\right)$ is contained in a rational linear subspace of dimension $n_2$, there is a unimodular transformation $u$ such that $u\left(\rec\left(\conv(S)\right)\right) = \rec\left(\conv(u(S))\right) \subseteq \left\{\mathbf{0}\right\}\times\mathbb{R}^{n_2}$.
Let $\rec\left(\conv(u(S))\right) = \cone\left\{r^1,\ldots,r^{h}\right\}$.
As $r^1,\ldots,r^{h}$ span $R^{n_2}$, it follows that $e^1,\ldots, e^{n_1}$, $r^1,\ldots,r^{h}$ span $\mathbb{R}^n$. Therefore, there exists a sufficiently large integer $M$ such that $v+M(\sum_{i=1}^{n_1}e^i+\sum_{j=1}^hr^j)\in \mathcal{C}$ for every vertex $v$ of $\conv(u(S))$.
Let $\nu$ be the undimodular transformation defined by $\nu(x):=x+M(\sum_{i=1}^{n_1}e^i+\sum_{j=1}^hr^j)$ for $x\in\mathbb{R}^n$. Then $\conv(\nu(u(S)))\subseteq \mathcal{C}$, and since $\nu$ is just a translation, the recession cone of $\conv(\nu(u(S)))$ remains the same as that of $\conv(u(S))$. Therefore, $\tau=\nu\circ u$ is the desired unimodular transformation.
\ifx\flagJournal\true \qed \fi
\end{proof}
By Lemma~\ref{LE:unimodular-closure}, $P_S$ is a rational polyhedron if and only if $\tau(P)_{\tau(S)}$ is a rational polyhedron, so we may assume that $S$ satisfies~\eqref{eq:S_pointed_assumption}. 

The second assumption is on the structure of the polyhedron $P$. Let $P^1$ and $P^2$ be defined as follows:
\begin{equation}\label{eq:P1P2}
P^1:=P+\mathcal{C},\quad P^2:=P-\mathcal{C}.
\end{equation}
Since $P\subseteq \conv(S)\subseteq \mathcal{C}$, $P^1$ is pointed and the extreme points of $P^1$ are contained in $\conv(S)$. Moreover, $P^1$ can be written as $P^1=\left\{x\in\mathbb{R}^n:\;Ax\geq b\right\}$
where $A\in\mathbb{Z}^{m\times n}$, $b\in\mathbb{Z}^m$ are matrices satisfying
\begin{equation}\label{eq:coeff-conditions}%
Ax\geq\mathbf{0}\text{ for all }x\in \left\{e^1,\ldots,e^{n_1},r^1,\ldots,r^{h}\right\}\quad\text{and}\quad b\geq\mathbf{0}.
\end{equation}
Similarly, $P^2$ can be written as $P^2=\left\{x\in\mathbb{R}^n:\;Ax\leq b\right\}$ for some $A\in\mathbb{Z}^{m\times n}$, $b\in\mathbb{Z}^m$ satisfying~\eqref{eq:coeff-conditions}. Basically, $P^1,P^2$ are polyhedra of the form $\pc$ or $\pp$:
\begin{equation}\label{eq:cov or pack}
\pc=\left\{x\in\mathbb{R}^n:\; Ax\geq b\right\}\quad\text{or}\quad \pp=\left\{x\in\mathbb{R}^n:\; Ax\leq b\right\}
\end{equation}
for some $A\in\mathbb{Z}^{m\times n}$, $b\in\mathbb{Z}^m$ satisfying~\eqref{eq:coeff-conditions}. When $\rec\left(\conv(S)\right)=\left\{{\bf 0}\right\}\times\mathbb{R}^{n_2}$, i.e., $\{r^1,\ldots, r^h\}=\{e^{n_1+1},\ldots, e^{n_1+n_2}\}$, $A$ and $b$ are simply matrices with nonnegative entries, in which case, $\pc$ is a covering polyhedron and $\pp$ is a packing polyhedron. In Sections~\ref{sec:cov} and~\ref{sec:pack}, we focus on polyhedra of the form $\pc$ and $\pp$, and we prove that the following holds:
\begin{theorem}\label{quote}
Let $Q\subseteq\mathbb{R}^n$ be a rational polyhedron of the form $\pc$ or $\pp$ as in~\eqref{eq:cov or pack} for some $A\in\mathbb{Z}^{m\times n}$, $b\in\mathbb{Z}^m$ satisfying~\eqref{eq:coeff-conditions}. Then $Q_{S}$ is a rational polyhedron provided that if $Q=\pc$, then $\pc\subseteq\mathcal{C}$ and the extreme points of $\pc$ are contained in $\conv(S)$. 
\end{theorem}
We will prove this theorem for $Q=\pc$ and $Q=\pp$ separately in Theorems~\ref{thm:cov} and~\ref{thm:pack}, respectively. The setting in Theorem~\ref{quote} is essentially the most difficult case, after settling which, the remaining step would be to provide a reduction from the case in which $P\subseteq \conv(S)$ is any rational polyhedron to the narrowed case in Theorem~\ref{quote}. We will show this reduction in Section~\ref{sec:proof}, for which the construction of $P^1$ and $P^2$ as in~\eqref{eq:P1P2} will show up. Although the condition in~Theorem~\ref{quote} that $\pc\subseteq\mathcal{C}$ and the extreme points of $\pc$ are contained in $\conv(S)$ might look arbitrary at first glance, $P^1$ in~\eqref{eq:P1P2} satisfies the condition because of the assumption that $P\subseteq \conv(S)\subseteq\mathcal{C}$.
\subsection{Covering polyhedra}\label{sec:cov}

In Sections~\ref{sec:cov} and~\ref{sec:pack}, we assume that $\conv(S)\cap\mathbb{Z}^n=S$ and $\conv(S)$ is of the form~\eqref{eq:S_pointed} satisfying~\eqref{eq:S_pointed_assumption}. In this section, we consider polyhedra of the form $\pc$ as in~\eqref{eq:cov or pack} 
where $A\in\mathbb{Z}^{m\times n}$ and $b\in\mathbb{Z}^m$ satisfy~\eqref{eq:coeff-conditions}. We will prove that if $\pc\subseteq\mathcal{C}$ and the extreme points of $\pc$ are contained in $\conv(S)$, then $\pc_S$ is a rational polyhedron. 

Notice that every valid inequality for $\pc$ is of the form
\begin{equation}\label{eq:coveq} 
\alpha x\geq \beta\quad\text{where~}\alpha x\geq 0\text{ for all }x\in\left\{e^1,\ldots,e^{n_1},r^1,\ldots,r^{h}\right\}\text{ and }\beta \geq0.
\end{equation}
In Section~\ref{sec:prelim}, we defined $\Pi_P$ to collect inequalities of the form $\alpha x\leq \beta$. However, as we will be dealing with inequalities of the form $\alpha x\geq \beta$  in this section, we will abuse notation and define $\Pi_{\pc} $ as follows:
\[
\Pi_{\pc} = \left\{ (\lambda A, \lambda b) \in \Z^n\times\R:\;\lambda\in\mathbb{R}_+^m,\;{\lambda b}= \min\{\lambda Ax : x \in \pc\}\right\}.\]
Given $(\alpha,\beta)\in \Pi_{\pc}$, the $S$-CG cut obtained from $\alpha x\geq \beta$ is $\alpha x\geq \ceil{\beta}_{S,\alpha}$. In Section~\ref{sec:cov} and~\ref{sec:pack}, we need the notion of ``ray-support" defined as follows. Given a vector $\alpha\in\mathbb{R}^n$, the {\it ray-support} of $\alpha$, denoted $r\text{-supp}(\alpha)$, is defined as
$$r\text{-supp}(\alpha):=\left\{j\in N_r:~\alpha r^j>0\right\}\quad \text{where} \ N_r=\{1,\ldots,h\}.$$
Hence, the ray-support of $\alpha$ indicates which rays among $r^1,\ldots, r^h$ in the recession cone of $\conv(S)$ intersect hyperplane $\{x\in\mathbb{R}^n:\alpha x=\beta\}$ for arbitrary $\beta$. If $(\alpha,\beta)\in \Pi_{\pc}$ and $j\in r\text{-supp}(\alpha)$, then $\alpha r^j\geq 1$ as $\alpha$ and $r^j$ have integer entries. For $j\in r\text{-supp}(\alpha)$, the ray generated by $r^j$ always intersects $\{x\in\mathbb{R}^n:\alpha x=\beta\}$ at $x=(\beta/\alpha r^j)r^j$. Henceforth, $\beta/\alpha r^j$ for $j\in r\text{-supp}(\alpha)$ is referred to as an \emph{intercept} of $(\alpha,\beta)$. Also, whenever we mention intercepts of the corresponding $S$-CG cut $\alpha x\geq \ceil{\beta}_{S,\alpha}$, they refer to intercepts of $(\alpha,\beta)$. Having defined $\Pi_{\pc}$ and the ray-support, we provide a brief outline of the proof.
\begin{itemize}
	\item[1.] (Lemma~\ref{LE:basics} and Theorem~\ref{thm:cov}) We show that every nondominated $S$-CG cut for $\pc$ has bounded intercepts. More precisely, if the $S$-CG cut derived from $(\alpha,\beta)\in \Pi_{\pc}$ is \emph{not} redundant, then $(\alpha,\beta)$ is contained in \begin{equation}\label{cov:Pi}
	\Pi=\left\{(\alpha,\beta)\in\Pi_{\pc}:\; \beta/\alpha r^j\leq M^*\text{ for all }j\in r\text{-supp}(\alpha)\right\}
	\end{equation}
	for some sufficiently large integer constant $M^*$.
	\item[2.] (Lemma~\ref{lemma:coverMstar}) We show that $\pc_{S,\Pi}$ is a rational polyhedron. As the first step implies that $\Pi$ collects all nondominated $S$-CG cuts for $\pc$, $\pc_{S,\Pi}=\pc_{S}$ and thus $\pc_{S}$ is also a rational polyhedron. Essentially, what Lemma~\ref{lemma:coverMstar} shows is, given that every nondominated \scg for $\pc$ has bounded intercepts, $\pc_S$ is a rational polyhedron.
\end{itemize}
We consider the second step first and prove the first step later. Before we proceed, let us state some high-level intuitions behind our approach. What does it mean that every nondominated $S$-CG cut has bounded intercepts? When all intercepts are bounded by a fixed constant \emph{and $r\text{-supp}(\alpha)=N_r$}, the intersection of $\{x\in\mathbb{R}^n:\alpha x=\beta\}$ and $\conv(S)$ is bounded and can be squeezed in a bounded polytope, in which there are finitely many integer points. Although $r\text{-supp}(\alpha)$ may be a proper subset of $N_r$, the idea is to reduce the problem to the finite case, for which we already have the following result:
\begin{theorem}[{\bf \cite[Theorem~2.7]{Dash19}}]\label{lemma:QcapPi} 
	Let $S$ be a finite subset of $\mathbb{Z}^n$ and $P\subseteq\mathbb{R}^n$ be a rational polyhedron. Let $H \subseteq \R^{n}\times\R$ be a rational polyhedron that is contained in its recession cone $\rec(H)$ and let $\Omega  = \Pi_P \cap H$. Then, $P_{S,\Omega }$ is a rational polyhedron.
\end{theorem}
For the first step, we take an $S$-CG cut with a large intercept. Then, starting from this cut, we construct a sequence of $S$-CG cuts, each cut in which is dominated by the next one and the cut at the end has bounded intercepts.

Now let us prove Lemma~\ref{lemma:coverMstar}. 
\begin{lemma}\label{lemma:coverMstar}
Let $\Pi$ be defined as in~\eqref{cov:Pi} for some positive integer $M^*$. Then $\pc_{S,\Pi}$ is a rational polyhedron.
\end{lemma}
\begin{proof}
Recall that $\conv(S)=\conv\left\{v^1,\ldots,v^{g}\right\}+\cone\left\{r^1,\ldots,r^{h}\right\}$. Let $S^*$ be a finite subset of $S$ defined as $$S^*:=S\cap(\conv\left\{v^1,\ldots,v^{g}\right\}+\left\{\mu_1r^1+\cdots+\mu_{h}r^{h}:0\leq \mu_j\leq M^*\text{ for }j\in N_r\right\}).$$ As $S^*\subseteq S$, we have $\pc_{S^*,\Pi}\subseteq \pc_{S,\Pi}$. We show that $\pc_{S^*,\Pi}=\pc_{S,\Pi}$, for which, it is sufficient to show that $\ceil{\beta}_{S^*,\alpha}=\ceil{\beta}_{S,\alpha}$ for every $(\alpha,\beta)\in \Pi$. To this end, take a pair $(\alpha,\beta)\in \Pi$. As $(\alpha,\beta)\in \Pi_{\pc}$ and $\conv(S)\subseteq \mathcal{C}$, it follows from~\eqref{eq:coveq} that $\alpha v^i\geq0$ for $i\in \{1,\ldots,g\}$, $\alpha r^j\geq 0$ for $j\in  N_r$, $\beta\geq 0$, and ${\beta}/{\alpha r^j}\leq M^*$ for $j\in r\text{-supp}(\alpha)$. 
Let $z^*\in S$ be such that $\alpha z^*=\ceil{\beta}_{S,\alpha}$.  As $z^*\in S\subseteq \conv(S)$, for some $\mu\geq\mathbf{0}$ and $\lambda\geq\mathbf{0}$ with $\mathbf{1}\lambda=1$, we have $z^*=\lambda_1v^1+\cdots+\lambda_{g}v^{g}+\mu_1r^1+\cdots+\mu_{h}r^{h}$. If $z^*\in S^*$, then $\alpha z^*=\ceil{\beta}_{S^*,\alpha}$, so $\ceil{\beta}_{S^*,\alpha}=\ceil{\beta}_{S,\alpha}$. Thus we may assume that $z^*\not\in S^*$, and therefore, there exists $j\in N_r$ with $\mu_j>M^*$. Let $\bar \mu$ be obtained from $\mu$ after reducing all coordinates of $\mu$ greater than $M^*$ to $M^*$. Let $\bar z$ be defined as $\bar z=\lambda_1v^1+\cdots+\lambda_{g}v^{g}+\bar \mu_1r^1+\cdots+\bar \mu_{h}r^{h}$. By definition, $\bar z\in S^*$. We will argue that $\alpha \bar z=\alpha z^*$, thereby showing that $\ceil{\beta}_{S^*,\alpha}=\ceil{\beta}_{S,\alpha}$. Since $\bar \mu\leq \mu$ and $\alpha r^j\geq 0$ for $j\in N_r$, we have $\alpha \bar z \leq \alpha z^*$. Suppose for a contradiction that $\alpha \bar z \neq \alpha z^*$. Then $\alpha \bar z<\alpha z^*$, and therefore, there exists $j\in N_r$ such that $\mu_j>M^*$ and $\alpha r^j>0$. Then $\bar \mu_j=M^*$, so $\alpha \bar z\geq \alpha \bar \mu_j r^j=M^*\alpha r^j$. As $\beta\leq M^*\alpha r^j$ for all $j\in r\text{-supp}(\alpha)$, this in turn implies that $\alpha \bar z\geq \beta$. However, this contradicts the choice of $z^*$ to be a minimizer of $\min\left\{\alpha z:\;\alpha z\geq \beta,\;z\in S\right\}$. Therefore, we have $\alpha z^*=\alpha z$, implying in turn that $\ceil{\beta}_{S^*,\alpha}=\ceil{\beta}_{S,\alpha}$ and that $\pc_{S^*,\Pi}=\pc_{S,\Pi}$.

Although $\Pi$ itself is not polyhedral, $\Pi(I)=\left\{(\alpha,\beta)\in\Pi:\;r\text{-supp}(\alpha)=I\right\}$ for any $I\subseteq  N_r$ is a rational polyhedron because $\Pi(I)=\Pi_{\pc}\cap H(I)$ where
\[
H(I)=\left\{(\alpha,\beta)\in\mathbb{R}^{n}\times\mathbb{R}:\;
\begin{array}{l}
\alpha r^j\geq 1\text{ for }j\in I,\ \ \alpha r^j=0\text{ for }j\in  N_r\setminus I,\\M^*\alpha r^j\geq\beta\text{ for }j\in I
\end{array}\right\}.
\]
Notice that $H(I)\subseteq \rec(H(I))$, so by Theorem~\ref{lemma:QcapPi}, $\pc_{S^*,\Pi(I)}$ is a rational polyhedron. As $\pc_{S^*,\Pi}=\bigcap_{I\subseteq N_r}\pc_{S^*,\Pi(I)}$, the proof is complete.
\ifx\flagJournal\true \qed \fi
\end{proof}

Next we go back to the first step and prove that the intercepts of every nondominated \scg for $\pc$ are bounded. Recall that $\pc$ is described by the system $Ax\geq b$ consisting of $m$ inequalities. We denote them by $a_1x\geq b_1,\ldots,a_mx\geq b_m$. Since $\pc$ is pointed, $m\geq 1$. Hence, for any $(\alpha,\beta)\in\Pi_{\pc}$, there is a multiplier vector $\lambda\in\mathbb{R}^m_+$ such that $\alpha=\lambda A=\sum_{i=1}^m\lambda_ia_i$ and $\beta=\lambda b$. Note that $a_ir^j\geq 0$ for all $i,j$, so for any $\lambda\in\mathbb{R}_+^m$, we have $r\text{-supp}(\lambda_ia_i)\subseteq r\text{-supp}(\lambda A)$. 
\begin{definition}\label{DE:maindef}
Let $\lambda\in\mathbb{R}^m_+\setminus\{\mathbf{0}\}$, and  $\lambda_{1}\geq\lambda_2\geq\cdots	\geq\lambda_{m}$. The {\it tilting ratio of $\lambda$ with respect to $A$} is defined as 
\begin{equation}\label{eq:t} 
r(\lambda,A) = {\lambda_{1}}/{\lambda_{t(\lambda, A)}}
\end{equation}
where  $t(\lambda, A)=\min\left\{j\in\{1, \ldots, m\}:\;\bigcup_{i=1}^jr\text{-supp}(a_i)=r\text{-supp}\left(\lambda A\right)\right\}$ is the smallest index $j$ such that the ray-support of $\sum_{i=1}^j \lambda_i a_i$ is the same as the ray-support of $\lambda A$.  In particular, $\lambda_1 \ldots, \lambda_{t(\lambda, A)} > 0$ and $r(\lambda,A)>0$.
\ifx\flagJournal\true \epr \fi
\end{definition}
It turns out that the tilting ratio is an important parameter for bounding the intercepts of an inequality. To demonstrate this, take $\lambda\in\mathbb{R}_+^m\setminus\left\{\mathbf{0}\right\}$ and $j\in r\text{-supp}(\lambda A)$. Note that	
\begin{equation}\label{eq:bounding}
\frac{\lambda b}{\lambda A r^j}=\frac{\sum_{i=1}^m\lambda_ib_i}{\sum_{i=1}^m\lambda_ia_ir^j}
\leq\frac{\lambda_1\sum_{i=1}^mb_i}{\lambda_t\sum_{i=1}^ta_ir^j}
=\frac{r(\lambda,A) \sum_{i=1}^mb_i}{\sum_{i=1}^ta_{i}r^j}\leq r(\lambda, A)\sum_{i=1}^mb_i
\end{equation}
where $t$ stands for $t(\lambda, A)$ and the last inequality is due to the fact that $\sum_{i=1}^ta_{i}r^j$ is a positive integer as $\bigcup_{i=1}^tr\text{-supp}\left(a_i\right)=r\text{-supp}\left(\lambda A\right)$. In~\eqref{eq:bounding}, $\sum_{i=1}^mb_i$ is fixed, which implies that if the tilting ratio $r(\lambda, A)$ is bounded, then the intercepts of $\lambda Ax\leq \lambda b$ are bounded. Therefore, it is sufficient to show that the tilting ratio of $\lambda$ is bounded. With this in mind, we focus on multiplier vectors $\lambda$ henceforth, instead of coefficient vectors $(\lambda A,\lambda b)$.

Next we decide the value of $M^*$ for~\eqref{cov:Pi}.
\begin{definition}\label{defM}\label{eq:B,C,D} 
Let $B=\max\limits_{1\leq i \leq m}\{b_i\}$ and $D=\sum_{i=1}^ma_i\left(\sum_{i=1}^{n_1}e^i+\sum_{j=1}^hr^j\right)$. We define $M_1 =2\left(mB+2D\right)$ and 
\begin{align*}
&M_i =(2mB\times M_1\times\cdots\times M_{i-1})^{i-1}M_1\ \ \text{for}\;i=2,\ldots,m-1.\label{eq:M_i}
\end{align*}
Having obtained $M_1,\ldots,M_{m-1}$, we define $M^*$ as follows:
$$M^*=mbM\quad\text{where}\quad M = \begin{cases}
M_1\times\cdots \times M_{m-1},&\text{if} \ m\geq 2\\
1,&\text{if} \ m=1
\end{cases}$$
In particular, if $m\geq 2$, $M\geq M_1\geq 4$. Moreover, $(M_i/M_1)^{1/(i-1)} \geq 4$, and thus, $(M_1/M_i)^{1/(i-1)} \leq 1/4$  for all $i \geq 2$. 
\ifx\flagJournal\true \epr \fi
\end{definition}
\noindent
By~\eqref{eq:bounding}, if $r(\lambda, A) \leq M$, the intercepts of $\lambda Ax\leq \lambda b$ are at most $M^*$ since $\sum_{i=1}^mb_i\leq mB$. What we will argue next is that when $r(\lambda, A)>M$, there exists another multiplier $\mu$ that defines an \scg dominating the one from $\lambda$. When $r(\lambda, A)$ is large, the components of $\lambda$ are not balanced in the sense that $\lambda_1$ is much larger than $\lambda_t$. In such case, we find a vector that approximates large components of $\lambda$, after substracting which from $\lambda$, we obtain a new multiplier $\mu$ that has more balanced components. Thanks to our choice of $M^*$ being a huge number in Definition~\ref{defM}, we will be able to show that the \scg derived from such multiplier $\mu$ dominates the one from $\lambda$. For the step of approximating the initial multiplier $\lambda$, we will need a result of Dirichlet:

\begin{theorem}[\bf Simultaneous Diophantine Approximation Theorem~\cite{D1842}]\label{apx}
	Let $k$ be a positive integer. Given any real numbers $r_1,\ldots,r_k$ and $0<\varepsilon<1$, there exist integers $p_1,\ldots,p_{k}$ and $q$ such that $\left|r_i-\frac{p_i}{q}\right|<\frac{\varepsilon}{q}$ for $i=1,\ldots,k$ and $1\leq q\leq\left(\frac{1}{\varepsilon}\right)^{k}$.
\end{theorem}

This idea of subtracting an approximate vector to construct another multiplier was first considered in~\cite{Dash19}, and the following lemma extends the idea to general polyhedra.
\begin{lemma}\label{LE:basics}
Let $\lambda \in \R^m_+\setminus\{\mathbf{0}\}$ be such that $(\lambda A,\lambda b)\in  \Pi_{\pc}$. If $r(\lambda, A) > M$, then there exists $\mu \in \R^m_+\setminus\{\mathbf{0}\}$ that satisfies the following: (i)  $\|\mu\|_1\leq \|\lambda\|_1-1$,  (ii) $(\mu A,\mu b)\in  \Pi_{\pc}$, and (iii) $\mu A x\geq \ceil{\mu b}_{S,\mu A}$ dominates $\lambda A x\geq \ceil{\lambda b}_{S,\lambda A}$.
\end{lemma}
\begin{proof}
After relabeling the rows of $Ax\geq b$, we may assume that $\lambda_1\geq\cdots\geq\lambda_m$. Let $t$ stand for $t(\lambda,A)$. If $t=1$, we have $r(\lambda,A)=1\leq M$, contradicting our assumption. This means that $t\geq 2$, so $m\geq 2$. Let $\Delta$ and $k$ be defined as
\begin{equation}\label{eq:delta}
\Delta=\min\left\{\lambda Ar^j:j\in r\text{-supp}(\lambda A)\right\},
\end{equation}
\begin{equation}\label{eq:k}
k=\text{argmin}\bigg\{\lambda Ar^j:\;j\in r\text{-supp}\left(\lambda A\right)\setminus\bigcup_{i=1}^{t-1}r\text{-supp}(a_i)\bigg\}.
\end{equation}
By the definition of $t$, $r\text{-supp}\left(\lambda A\right)\setminus\bigcup_{i=1}^{t-1}r\text{-supp}(a_i)$ is not empty, and therefore, $k$ is a well-defined index. Moreover, we obtain
\begin{equation}\label{eq:basics-1}
\Delta\leq \lambda Ar^k=\sum_{i=t}^m\lambda_ia_ir^k\leq \lambda_t\sum_{i=t}^ma_ir^k\leq D\lambda_t
\end{equation}
where the first inequality is due to~\eqref{eq:delta}, the equality holds due to~\eqref{eq:k}, the second inequality follows from the assumption that $\lambda_t\geq\lambda_{t+1}\geq\cdots\geq\lambda_m$, and the last inequality follows from the choice of $D$ given in Definition~\ref{eq:B,C,D}.
As $r(\lambda, A) = \frac{\lambda_1}{\lambda_t} = \frac{\lambda_1}{\lambda_2} \times \cdots \times \frac{\lambda_{t-1}}{\lambda_{t}} > M \ge M_1 \times \cdots \times M_{t-1}$, 
there exists  $\ell\in \{1, \ldots, t-1\}$ such that $\lambda_{\ell} / \lambda_{\ell+1}>M_\ell$. We take the minimum number among such indices, so we may assume that
\begin{equation}\label{eq:ell}
\lambda_{i} / \lambda_{i+1} \leq M_i\text{ for all }i \in \{ 1, \ldots, \ell-1\}\quad\text{and}\quad	\lambda_{\ell} / \lambda_{\ell+1}>M_\ell.
\end{equation}
Hence, $\lambda_1,\ldots,\lambda_\ell$ are much larger than $\lambda_{\ell+1},\ldots, \lambda_{t}$, as $M_1,\ldots, M_{t-1}$ were chosen to be large numbers in Definition~\ref{defM}. Now we construct the vector $\mu\in\R^m\setminus\{\mathbf{0}\}$. We consider the case $\ell\geq 2$ first. It follows from Theorem~\ref{apx} (with $k = \ell-1$ and $r_i = \lambda_i/\lambda_\ell$ for $i\in\{1, \ldots, \ell-1\}$) that there exist positive integers $p_1,\ldots,p_\ell$ satisfying
\begin{equation}\label{eq:Diophantine-approx}
\left|{\lambda_i}/{\lambda_{\ell}}-{p_i}/{p_{\ell}}\right|<{\varepsilon}/{p_{\ell}},\;i\in\{1,\ldots,\ell\}
\quad\text{and}\quad p_{\ell}\leq\varepsilon^{-(\ell-1)}
\end{equation}
where $\varepsilon = (M_1/M_\ell)^{1/(\ell-1)}$.
Moreover, for all $i \in\{1, \ldots, \ell-1\}$, we can assume that $p_i \geq p_{i+1} \geq p_\ell$, as $\lambda_i \geq \lambda_{i+1}$. 
If $p_i < p_{i+1}$ for some $i \in \{1, \ldots, \ell-1\}$, then increasing $p_i$ to $p_{i+1}$ can only reduce $|\lambda_i/\lambda_\ell - p_i/p_\ell|$. Note that $(p_1,\ldots,p_\ell,0,\ldots,0)\in \R_+^m$ is a vector approximating the large components of $\lambda$. Now we define a new multiplier $\mu=(\mu_1, \ldots, \mu_m)$ by taking out a multiple of $(p_1,\ldots,p_\ell,0,\ldots,0)$ from $\lambda$ as follows:
\begin{equation}\label{defmu}
\mu_i = \left\{\begin{array}{ll} \lambda_i - p_i \Delta & \text{ for } i\in\{1, \ldots, \ell\},\\
\lambda_i & \text{ otherwise}.
\end{array}\right.
\end{equation}
If, on the other hand, $\ell=1$, we  define $\mu$ as in \eqref{defmu} with $p_1=1$. 

Next, we show that $\mu$ satisfies the desired properties. First of all, Claim~1 below can be proved similarly as Claim~1 in the proof of Lemma 4.10 in \cite{Dash19}.
\begin{claim'}\label{claim2}
	$\mu\ge\mathbf{0}$ and $\text{supp}(\mu )=\text{supp}(\lambda)$.
\end{claim'}
\noindent
Since $\text{supp}(\mu)=\text{supp}(\lambda)$ by Claim~1 and $Ar^j\geq \mathbf{0}$ for all $j\in N_r$, it follows that $r\text{-supp}(\mu A)=r\text{-supp}(\lambda A)$, and therefore, $t(\mu, A) = t(\lambda, A)$. 

The next claims extend Claims 2--4 of Lemma 4.10 in~\cite{Dash19}.
\begin{claim'}\label{claim4}
	$\mu b=\min\left\{\mu Ax:x\in\pc\right\}$ and therefore $(\mu A,\mu b)\in  \Pi_{\pc}$.
\end{claim'}
{\textit{Proof of Claim.}} As $\lambda b=\min\left\{\lambda Ax:x\in\pc\right\}$ and $\pc=\left\{x\in\mathbb{R}^n:Ax\geq b\right\}$, there exists $x^*\in \pc$ such that $\lambda A^*x = \lambda b$. 
By complementary slackness, if $\lambda_i>0$ for an $i\in\{1, \ldots, m\}$, then $a_ix^*=b_i$. 
As $\lambda\ge\mu\ge \mathbf{0}$, 
if $\mu_i>0$ then $a_ix^*=b_i$ also holds. 
Therefore, $\mu A x^*=\mu b=\min\left\{\mu Ax:x\in\pc\right\}$.
\ifx\flagJournal\true \epr \fi
\begin{claim'}\label{claim5}
Let $\Theta = \left\{x \in \mathcal{C}:\;\mu b \leq \mu A x \leq \mu b  + \Delta\right\}$. There is no point $x\in \Theta$ that satisfies
\begin{equation}\label{ineqcontra}
\sum_{i=1}^{\ell} p_ia_ix \geq 1+ \sum_{i=1}^{\ell} p_ib_i.
\end{equation}
\end{claim'}
{\textit{Proof of Claim.}} Suppose for a contradiction that there exists $\tilde x \in \Theta$ satisfying~\eqref{ineqcontra}. Recall that for the index $k$ defined in \eqref{eq:k}, the inequality $\mu Ar^k>0$ holds. 
Let $v = \frac{\mu b}{\mu Ar^k}r^k$. Then $\mu A v =\mu b$ and $v\in \Theta$. In addition, for the index $\ell$ defined in \eqref{eq:ell}, we have $\sum_{i=1}^\ell p_ia_iv =0$
since $k\not\in \bigcup_{i=1}^{t-1}r\text{-supp}(a_i)$ and $a_ir^k=0$ for  $i\leq t-1$. As  $\tilde x \in \Theta$ satisfies \eqref{ineqcontra} and $v\in \Theta$ satisfies $\sum_{i=1}^\ell p_ia_iv =0$, we can take a convex combination of these points to get a point $\bar x \in \Theta$ such that $\sum_{i=1}^\ell p_ia_i \bar x=1+\sum_{i=1}^\ell p_ib_i$ and thus $\sum_{i=1}^\ell p_i(a_i \bar x-b_i) =1$. As $\mu A \bar x  \le \mu b + \Delta$, we have
\begin{equation}\label{eq:5star}
\sum_{i=1}^\ell\mu_i(a_i\bar x-b_i) \le  -\sum_{j=\ell+1}^m\mu_j(a_j\bar x-b_j) + \Delta.
\end{equation}
By~\eqref{eq:Diophantine-approx}, we can define $\varepsilon_i \in[ -\varepsilon , \varepsilon]$ such that ${\lambda_i}/{\lambda_\ell}-{p_i}/{p_\ell}={\varepsilon_i}/{p_\ell}$.
		Then, along with the fact that $\mu_i=\lambda_i-p_i\Delta$ for $i\leq \ell$ and $\sum_{i=1}^\ell p_i(a_i \bar x-b_i) =1$, we can rewrite the left hand side of~\eqref{eq:5star} as $(\frac{\lambda_\ell}{p_\ell} - \Delta) + \frac{\lambda_\ell}{p_\ell}\sum_{i=1}^\ell \varepsilon_i(a_i\bar x-b_i)$.
		Therefore, we deduce from~\eqref{eq:5star} that
		\begin{align} &\frac{\lambda_\ell}{p_\ell}\bigg(1 + \sum_{i=1}^\ell\varepsilon_i(a_i\bar x-b_i)\bigg) \leq -\sum_{j=\ell+1}^m\mu_j(a_j\bar x-b_j) + 2\Delta\notag\\&\qquad\qquad\qquad\leq \sum_{j=\ell+1}^m\mu_jb_j + 2\Delta \leq \lambda_{\ell+1}(mB + 2D) = \frac{1}{2}\lambda_{\ell+1}M_1 \label{eq:double-star}
		\end{align}
		where the second inequality in~\eqref{eq:double-star} follows from the assumption that $A\in\mathbb{Z}^{m\times n}$ and $b\in\mathbb{Z}^m$ satisfy~\eqref{eq:coeff-conditions} and the third inequality follows from the fact that $\mu_i = \lambda_i \leq \lambda_{\ell+1}$ for $i=\ell+1, \ldots, m$ by~\eqref{defmu}, $b_j \leq B$ by Definition~\ref{defM}, and $\Delta\leq D\lambda_t$ in~\eqref{eq:basics-1}. The last equality  simply follows from the definition of $M_1$.
		
		Next, we obtain a lower bound the first term in~\eqref{eq:double-star}. As $a_i\bar x\geq0$, $b_i\geq0$, and $\varepsilon_i \in[ -\varepsilon , \varepsilon]$, we have $\sum_{i=1}^\ell\varepsilon_i\left(a_i\bar x-b_i\right) \geq - \varepsilon\sum_{i=1}^\ell (a_i\bar x + b_i)$.
		Following the same argument in Claim~3 of Lemma 4.10 in~\cite{Dash19}, we can show that $-\varepsilon\sum_{i=1}^\ell (a_i\bar x + b_i) \geq -\frac{1}{2}$.
		Then it follows from that $\sum_{i=1}^\ell\varepsilon_i(a_i\bar x-b_i)\geq -{1}/{2}$. So, the first term of~\eqref{eq:double-star} is lower bounded by ${\lambda_\ell}/{2p_\ell}$. Since the first term in~\eqref{eq:double-star} is at least ${\lambda_\ell}/{2p_\ell}$, we obtain $\lambda_\ell\le p_\ell \lambda_{\ell+1} M_1$ from~\eqref{eq:double-star}, implying in turn that $M_\ell< p_\ell M_1$ as we assumed that $\lambda_\ell> M_\ell\lambda_{\ell+1}$ in~\eqref{eq:ell}. However,~\eqref{eq:Diophantine-approx} implies that $M_\ell\geq p_\ell M_1$, a contradiction.
		\ifx\flagJournal\true \epr \fi

	\begin{claim'}\label{claim7}
		$\mu A x\geq \ceil{\mu b}_{S,\mu A}$ dominates $\lambda A x\geq \ceil{\lambda b}_{S,\lambda A}$.	
	\end{claim'}
		{\textit{Proof of Claim.}} We will first show that	\begin{equation}\label{eq:claim6} \mu b\leq \ceil{\mu b}_{S,\mu A}\leq \mu b+\Delta\end{equation} holds.
		Set $(\alpha,\beta)=(\mu A,\mu b)$.
		By Claim~\ref{claim4}, we have that $\beta =\min\{\alpha x:x\in \pc\}$. As the extreme points of $\pc$ are contained in $\conv(S)$, it follows that $\beta \geq \min\{\alpha z:z\in S\}$. If $\beta =\min\{\alpha z:z\in S\}$, then $\beta = \ceil{\beta }_{S,\alpha}$. 
		Thus we may assume that $\beta >\min\{\alpha z:z\in S\}$, so there exists $z'\in S$ such that $\beta >\alpha z'$.  Remember that $\Delta=\min\{\lambda Ar^j:j\in r\text{-supp}(\lambda A)\}$ in~\eqref{eq:delta}. Take $j$ such that $\lambda Ar^j=\Delta$.
		As $r\text{-supp}(\lambda A)=r\text{-supp}(\mu A)$, we have $\alpha r^j>0$ and $\kappa=({\beta - \alpha  z'})/{\alpha r^j}>0$. Therefore $z''=z'+\lceil\kappa\rceil r^j\in S$. 
		Observe that 
		$\beta =\alpha \left(z'+\kappa r^j\right)
		\leq \alpha \left(z'+\lceil\kappa\rceil r^j\right)  
		=\beta+\alpha r^j(\lceil\kappa\rceil-\kappa)
		\leq~\beta+\alpha r^j$.
		As $\lambda\ge\mu$, we have $\Delta\ge \alpha r^j$ implying $\beta\le\alpha z''\le\beta+\Delta$ and  \eqref{eq:claim6} hold, as desired.
		
		Using~\eqref{eq:claim6}, we will show that $\mu A x\geq \ceil{\mu b}_{S,\mu A}$ dominates $\lambda A x\geq \ceil{\lambda b}_{S,\lambda A}$. Let $z\in S$ be such that $\mu A z=\ceil{\mu b}_{S,\mu A}$. 
		As $z$ is integral and $\mu b\leq \ceil{\mu b}_{S,\mu A}\leq \mu b+\Delta$ by~\eqref{eq:claim6}, Claim~\ref{claim5} implies that $\sum_{i=1}^\ell p_ia_i z < 1 + \sum_{i=1}^\ell p_ib_i$ and thus $\sum_{i=1}^\ell p_ia_i z = \sum_{i=1}^\ell p_ib_i - f$
		for some integer $f \in [0,\sum_{i=1}^\ell p_ib_i]$. Consider $z+fr^j\in S$ and note that $\lambda A\left(z+fr^j\right)=\left(\mu A+\Delta\sum_{i=1}^\ell p_ia_i\right)z+\Delta\sum_{i=1}^\ell p_i(b_i - a_iz)=\ceil{\mu b}_{S,\mu A}+\Delta\sum_{i=1}^\ell p_ib_i$.
		Since $\ceil{\mu b}_{S,\mu A} \geq \mu b$, we must have $\ceil{\mu b}_{S,\mu A}+\Delta\sum_{i=1}^\ell p_ib_i\geq\mu b+\Delta\sum_{i=1}^\ell p_ib_i=\lambda b.$
		Then $\ceil{\mu b}_{S,\mu A}+\Delta\sum_{i=1}^\ell p_ib_i\geq \ceil{\lambda b}_{S,\lambda A}$. So, the inequality $\lambda A x\geq \ceil{\lambda b}_{S,\lambda A}$ is dominated by $\mu A x\geq \ceil{\mu b}_{S,\mu A}$, as the former is implied by the latter and a nonnegative combination of the inequalities in $Ax\geq b$, as required.
		\ifx\flagJournal\true \epr \fi

By construction, $\mu$ satisfies (i), and by Claims~\ref{claim4} and~\ref{claim7}, $\mu$ satisfies (ii) and (iii), as required.
\ifx\flagJournal\true \qed \fi
\end{proof}

Now we are ready to prove that $\pc_S$ is a rational polyhedron. 
\begin{theorem}\label{thm:cov} 
Let $\Pi$ be defined as in~\eqref{cov:Pi} with $M^*=mBM$. If $\pc\subseteq\mathcal{C}$ and the extreme points of $\pc$ are contained in $\conv(S)$, then $\pc_S=\pc_{S,\Pi}$, and in particular, $\pc_S$ is a rational polyhedron.
\end{theorem}
\begin{proof}
As $\Pi\subseteq\Pi_{\pc}$, we have $\pc_{S}\subseteq \pc_{S,\Pi}$.
We will show that $\pc_{S}=\pc_{S,\Pi}$ by arguing that for each $(\alpha,\beta)\in \Pi_{\pc}$, there is an $(\alpha',\beta')\in \Pi$ such that the \scg  derived from  $(\alpha',\beta')$ dominates the \scg  derived from  $(\alpha,\beta)$ on $\pc$ by constructing a sequence that ends with such $(\alpha',\beta')$.

Let $\lambda\in\mathbb{R}_+^m\setminus\left\{\mathbf{0}\right\}$ be such that  $(\lambda A,\lambda b)\in \Pi_{\pc}$, and set $(\alpha,\beta)=(\lambda A,\lambda b)$. If ${\beta}/{\alpha r^j}\leq M^*$   for all $j\in r\text{-supp}(\alpha)$, then $(\alpha,\beta)\in\Pi$ as desired. Otherwise, consider an arbitrary  $j\in r\text{-supp}(\alpha)$ such that ${\beta}/{\alpha r^j} >  M^*$. By~\eqref{eq:bounding}, we obtain $M^*<{mB}\,r(\lambda, A)$.
As $M^* = mBM$, we have  $r(\lambda, A) > M$. 
Then, by Lemma~\ref{LE:basics}, there exists a $\mu\in \R^m_+\setminus\left\{\mathbf{0}\right\}$ such that $\|\mu\|_1\leq \|\lambda\|_1-1$ and the \scg  generated by $\mu$ dominates the \scg  generated by $\lambda$ for $\pc$.
If necessary,  we can repeat this argument and construct a sequence of vectors $\mu^1, \mu^2, \ldots,$ with decreasing norms, each of which defines an \scg that dominates the previous one. 
Therefore, after at most $\|\lambda\|_1$ iterations, we must obtain a vector $\hat\mu\in\R^m_+\setminus\left\{\mathbf{0}\right\}$ such that $r(\hat\mu, A) \leq M$ and  $(\hat\mu A,\hat\mu b)\in \Pi$. 
As $(\hat\mu A,\hat\mu b)\in\Pi$ and the \scg  generated by $\hat\mu$ dominates the \scg  generated by $\lambda$ for $\pc$, we conclude that  $\pc_S=\pc_{S,\Pi}$. 
Moreover, as  $\pc_{S,\Pi}$ is a rational polyhedron by Lemma~\ref{lemma:coverMstar}, it follows that $\pc_S$ is a rational polyhedron, as desired.
\ifx\flagJournal\true \qed \fi
\end{proof}
\subsection{Packing polyhedra}\label{sec:pack}

In this section, we show that $\pp_S$ is a rational polyhedron, where $\pp$ is defined as in~\eqref{eq:cov or pack} for some $A\in\mathbb{Z}^{m\times n}$ and $b\in\mathbb{Z}^m$ satisfying~\eqref{eq:coeff-conditions}.  Unlike $\pc$, $\pp$ is not necessarily pointed. Another difference is that we do not need to assume that the extreme points of $\pp$ are contained in $\conv(S)$. Other than these, intuitions and techniques developed for $\pc$ still apply to $\pp$ as well. If $\pp=\mathbb{R}^n$, then $\pp_S=\mathbb{R}^n$ is trivially a rational polyhedron. Hence, we may assume that $m\geq 1$. As in~\eqref{eq:pip}, we define $\Pi_{\pp}$ as 
\begin{eqnarray}
\Pi_{\pp} = \left\{ (\lambda A,\lambda b) \in \Z^n\times\R:\;\lambda\in\mathbb{R}_+^m,\;{\lambda b}= \max\{\lambda A x : x \in \pp\}\right\}.
\end{eqnarray}
Given $(\alpha,\beta)\in \Pi_{\pp}$, the $S$-CG cut obtained from $\alpha x\leq \beta$ is $\alpha x\leq \floor{\beta}_{S,\alpha}$. 

\begin{lemma}\label{lemma:packMstar}
Let $M^*$ be a positive integer, and let
\begin{equation}
\Pi=\left\{(\alpha,\beta)\in\Pi_{\pp}:\; {\beta}/{\alpha r^j}\leq M^*\text{ for all }j\in r\text{-supp}(\alpha)\right\}.
\end{equation}
Then $\pp_{S,\Pi}$ is a rational polyhedron.
\end{lemma}
\begin{proof}
The proof is very similar to that of Lemma~\ref{lemma:coverMstar}. 
Let $S^*$ be a finite subset of $S$ defined as $$S^*:=S\cap(\conv\left\{v^1,\ldots,v^{g}\right\}+\left\{\mu_1r^1+\cdots+\mu_{h}r^{h}:0\leq \mu_j\leq M^*\text{ for }j\in N_r\right\}).$$ As $S^*\subseteq S$, $\pp_{S^*,\Pi}\subseteq \pp_{S,\Pi}$. To show that $\pp_{S^*,\Pi}=\pp_{S,\Pi}$, we will argue that $\floor{\beta}_{S^*,\alpha}=\floor{\beta}_{S,\alpha}$ for every $(\alpha,\beta)\in\Pi$. To this end, take an $(\alpha,\beta)\in\Pi$. 
Let $z^*\in S$ be such that $\alpha z^*=\floor{\beta}_{S,\alpha}$.  As $z^*\in S\subseteq \conv(S)$, for some $\mu\geq\mathbf{0}$ and $\lambda\geq\mathbf{0}$ with $\mathbf{1}\lambda=1$, we have $z^*=\lambda_1v^1+\cdots+\lambda_{g}v^{g}+\mu_1r^1+\cdots+\mu_{h}r^{h}$. If $z^*\in S^*$, then $\alpha z^*=\floor{\beta}_{S^*,\alpha}$, so $\floor{\beta}_{S^*,\alpha}=\floor{\beta}_{S,\alpha}$. Thus we may assume that $z^*\not\in S^*$, and therefore, there exists $j\in N_r=\{1,\ldots,h\}$ with $\mu_j>M^*$. Let $\bar \mu$ be obtained from $\mu$ after reducing all coordinates of $\mu$ greater than $M^*$ to $M^*$. Let $\bar z$ be defined as $\bar z=\lambda_1v^1+\cdots+\lambda_{g}v^{g}+\bar \mu_1r^1+\cdots+\bar \mu_{h}r^{h}$. By definition, $\bar z\in S^*$. As in the proof of Lemma~\ref{lemma:coverMstar}, it can be shown that $\alpha z^*=\alpha \bar z$, implying in turn that $\floor{\beta}_{S^*,\alpha}=\floor{\beta}_{S,\alpha}$ and that $\pp_{S^*,\Pi}=\pp_{S,\Pi}$.

Notice that $\Pi=\bigcup_{I\subseteq  N_r}\Pi(I)$
where $\Pi(I)=\left\{(\alpha,\beta)\in\Pi:\;r\text{-supp}(\alpha)=I\right\}$ and that $\Pi(I)=\Pi_{\pp}\cap H(I)$ where
\[
H(I)=\left\{(\alpha,\beta)\in\mathbb{R}^{n}\times\mathbb{R}:\;
\begin{array}{l}
\alpha r^j\geq 1\text{ for }j\in I,\ \ \alpha r^j=0\text{ for }j\in  N_r\setminus I,\\M^*\alpha r^j\geq\beta\text{ for }j\in I
\end{array}\right\}.
\]
As $H(I)\subseteq \rec(H(I))$, Theorem~\ref{lemma:QcapPi} implies that $\pp_{S^*,\Pi(I)}$ is a rational polyhedron. So, as $\pp_{S^*,\Pi}=\bigcap_{I\subseteq N_r}\pp_{S^*,\Pi(I)}$, the proof is complete.
\ifx\flagJournal\true \qed \fi
\end{proof}

The following lemma is the analogue of Lemma~\ref{LE:basics} for $\pc$, whose proof is almost identical to that of Lemma~\ref{LE:basics}. The difference is that here we consider inequalities of the form $\lambda A x\leq \lambda b$ and we decrease the right-hand side to obtain an $S$-CG cut, and as a result, we focus on integer points ``below" the hyperplane $\{x\in\R^n:\lambda Ax=\lambda b\}$, i.e., integer points $z$ such that $\lambda A z\leq \lambda b$. This lets us not assume that the extreme points of $\pp$ are contained in $\conv(S)$. Given $\lambda\in\mathbb{R}_+^m\setminus\{\mathbf{0}\}$, as in Definition~\ref{DE:maindef}, we can define the tilting ratio of $\lambda$ with respect to $A$, and we denote it by $r(\lambda, A)$. In addition, we define $B,D$, $M_i$ for $i\in\{1,\ldots,m-1\}$, $M$, and $M^*$ as in Definition~\ref{defM}.

\begin{lemma}\label{LE:pbasics}
Let $\lambda \in \R^m_+\setminus\{\mathbf{0}\}$ be such that $(\lambda A,\lambda b)\in  \Pi_{\pp}$. If $r(\lambda, A) > M$, then there exists $\mu \in \R^m_+\setminus\{\mathbf{0}\}$ that satisfies the following: (i)  $\|\mu\|_1\leq \|\lambda\|_1-1$,  (ii) $(\mu A,\mu b)\in  \Pi_{\pp}$, 
and (iii) $\mu A x\leq \floor{\mu b}_{S,\mu A}$ dominates $\lambda A x\leq \floor{\lambda b}_{S,\lambda A}$.
\end{lemma}
\begin{proof}
After relabeling the rows of $Ax\leq b$, we may assume that $\lambda_1\geq\cdots\geq\lambda_m$. Let $t(\lambda,A)$ be defined as in Definition~\ref{DE:maindef}, and let $t$ stand for $t(\lambda,A)$. If $t=1$, we have $r(\lambda,A)=1\leq M$, a contradiction to our assumption. So, $t\geq 2$, which implies that $m\geq 2$. Let $\Delta$ and $k$ be defined as in~\eqref{eq:delta} and~\eqref{eq:k}. As $r\text{-supp}\left(\lambda A\right)\setminus\bigcup_{i=1}^{t-1}r\text{-supp}(a_i)$ is not empty, it follows that $k$ is a well-defined index. Moreover, as $r(\lambda, A) > M_1 \times \cdots \times M_{m-1}$, there exists some $\ell\in \{1, \ldots, t-1\}$ such that~\eqref{eq:ell} is satisfied. As in the proof of Lemma~\ref{LE:basics}, we now construct another multiplier $\mu\in\mathbb{R}^m$. We follow the same route of subtracting a vector that approximates large components of $\lambda$ to construct $\mu$.

Let us first consider the $\ell\geq 2$ case. By Theorem~\ref{apx} (with $k = \ell-1$ and $r_i = \lambda_i/\lambda_\ell$ for $i\in\{1,\ldots,k\}$), there exist positive integers $p_1,\ldots,p_\ell$ that satisfy \eqref{eq:Diophantine-approx}. $\mu$ is defined as follows:
\begin{equation}\label{pdefmu}
\mu_i = \left\{\begin{array}{ll} \lambda_i - p_i \Delta & \text{ for } i=1, \ldots, \ell,\\
\lambda_i & \text{ otherwise}
\end{array}\right.
\end{equation}
Even for the case $\ell=1$, let $\mu$ be defined as in~\eqref{pdefmu} with $p_1=1$. As before, we can show that $\mu\geq\mathbf{0}$, $\supp(\mu)=\supp(\lambda)$ and $\mu b=\max\left\{\mu Ax:x\in\pp\right\}$ and therefore $(\mu A,\mu b)\in\Pi_{\pp}$. Moreover, it follows from $\mu\geq\mathbf{0}$ and~\eqref{pdefmu} that $\|\mu\|_1\leq \|\lambda\|_1-1$.
	
We next define $\Theta := \left\{x \in\mathcal{C}:\mu b-\Delta\leq \mu A x \leq \mu b\right\}$ and show that there is no point $x\in \Theta$ that satisfies
\begin{equation}\label{pineqcontra}
\sum_{i=1}^{\ell} p_ia_ix \geq 1+ \sum_{i=1}^{\ell} p_ib_i.
\end{equation}
Note that this $\Theta$ is defined differently than the one defined in Claim \ref{claim5} of Lemma~\ref{LE:basics}. Now $\Theta$ collects $x$ satisfying $\mu b-\Delta\leq \mu A x \leq \mu b$ instead of $\mu b\leq \mu A x \leq \mu b+\Delta$. Suppose for a contradiction that there exists $\tilde x \in \Theta$ satisfying~\eqref{pineqcontra}. Taking a convex combination of $\tilde x$ with the point $v=\frac{\mu b}{\mu A r^k}r^k\in \Theta$, we can construct $\bar x \in \Theta$ such that $\sum_{i=1}^\ell p_ia_i \bar x=1+\sum_{i=1}^\ell p_ib_i$. As $\bar x\in \Theta$, we have $\mu A\bar x\leq \mu b$, which can be rewritten as $\sum_{i=1}^\ell\mu_i(a_i\bar x-b_i) \leq -\sum_{j=\ell+1}^m\mu_j(a_j\bar x-b_j)$. As $\Delta>0$, it follows that
\begin{equation}\label{eq:pstar}
\sum_{i=1}^\ell\mu_i(a_i\bar x-b_i) \leq -\sum_{j=\ell+1}^m\mu_j(a_j\bar x-b_j) + \Delta.
\end{equation}
Note that inequality  \eqref{eq:pstar} is the same as \eqref{eq:5star}. The same argument used for proving Claim~\ref{claim5} of Lemma~\ref{LE:basics} can be repeated to obtain the desired contradiction.

Finally, to show that $\lambda A x\leq \floor{\lambda b}_{S,\lambda A}$ is implied by $\mu A x\leq \floor{\mu b}_{S,\mu A}$ and the inequalities in $Ax\leq b$, we first show that
\begin{equation}\label{eq:6claim}
\mu b-\Delta\leq \floor{\mu b}_{S,\mu A}\leq \mu b
\end{equation} holds. Set $(\alpha,\beta)=(\mu A,\mu b)$. There exists $z\in S$ such that $\alpha z=\floor{\beta}_{S,\alpha}$. Recall that by~\eqref{eq:delta}, $\Delta=\min\{\lambda Ar^j:j\in r\text{-supp}(\lambda A)\}$, and let $j$ be such that $\lambda A r^j=\Delta$. As $z+r^j\in S$ and $\alpha z=\floor{\beta}_{S,\alpha}$, it follows that $\alpha(z+r^j)=\floor{\beta}_{S,\alpha}+\alpha r^j>\floor{\beta}_{S,\alpha}$. That means $\alpha (z+r^j)>\beta$. Hence, we obtain $\floor{\beta}_{S,\alpha}+\alpha r^j>\beta$, which implies that $\floor{\beta}_{S,\alpha}\geq \beta-\alpha r^j\geq \beta-\Delta$, as required.

There exists $z\in S$ such that $\mu A z=\floor{\mu b}_{S,\mu A}$, and \eqref{eq:6claim} implies that $\mu b-\Delta\leq \mu Az\leq \mu b$. Since we have shown that there is no point $x\in \Theta$ satisfying~\eqref{pineqcontra}, it follows that $\sum_{i=1}^\ell p_ia_i z = \sum_{i=1}^\ell p_ib_i - f$ for some integer $f \in \left[0,\sum_{i=1}^\ell p_ib_i\right]$. Note that $\lambda A\left(z+fr^j\right)=\left(\mu A+\Delta\sum_{i=1}^\ell p_ia_i\right)z+\Delta\sum_{i=1}^\ell p_i(b_i - a_iz)=\floor{\mu b}_{S,\mu A}+\Delta\sum_{i=1}^\ell p_ib_i$. Since $\floor{\mu b}_{S,\mu A} \leq \mu b$, we must have $\floor{\mu b}_{S,\mu A}+\Delta\sum_{i=1}^\ell p_ib_i\leq\mu b+\Delta\sum_{i=1}^\ell p_ib_i=\lambda b$. Then $\floor{\mu b}_{S,\mu A}+\Delta\sum_{i=1}^\ell p_ib_i\leq \floor{\lambda b}_{S,\lambda A}$. So, the inequality $\lambda A x\leq \floor{\lambda b}_{S,\lambda A}$ is dominated by $\mu A x\leq \floor{\mu b}_{S,\mu A}$, as the former is implied by the latter and a nonnegative combination of the inequalities in $Ax\leq b$, as required.
\ifx\flagJournal\true \qed \fi
\end{proof}

Using Lemmas~\ref{lemma:packMstar} and~\ref{LE:pbasics}, we next prove that $\pp_S$ is a rational polyhedron.
\begin{theorem}\label{thm:pack} 
Let $\Pi=\{(\alpha,\beta)\in\Pi_{\pp}: {\beta}/{\alpha r^j}\leq M^*\text{ for all $j\in r\text{-supp}(\alpha)$}\}$ where $M^*=mBM$. Then $\pp_S=\pp_{S,\Pi}$, and $\pp_S$ is a rational polyhedron.
\end{theorem}
\begin{proof}
Recall that $\pp_S=\pp_{S,\Pi_{\pp}}$ by~\eqref{eq:pip}. As $\Pi\subseteq\Pi_{\pp}$, we have $\pp_{S,\Pi_{\pp}}\subseteq \pp_{S,\Pi}$. To show that $\pp_{S,\Pi_{\pp}}=\pp_{S,\Pi}$, we argue that for each $(\alpha,\beta)\in \Pi_{\pp}$ there is an $(\alpha',\beta')\in \Pi$ such that the \scg derived from  $(\alpha',\beta')$ dominates the \scg derived from  $(\alpha,\beta)$ on $\pp$.
	
Let $\lambda\in\mathbb{R}_+^m\setminus \{\mathbf{0}\}$ be such that  $(\lambda A,\lambda b)\in \Pi_{\pp}$ and let $(\alpha,\beta) = (\lambda A,\lambda b)$.
If ${\beta}/{\alpha r^j}\leq M^*$   for all $j\in r\text{-supp}(\alpha)$, then $(\alpha,\beta)\in\Pi$ as desired.
Otherwise, consider an arbitrary  $j\in r\text{-supp}(\alpha)$ such that ${\beta}/{\alpha r^j} >  M^*$. As we argued in the proof of Theorem~\ref{thm:cov}, it can be shown that $M^*<mBr(\lambda,A)$. As $M^* = mBM$, we have  $r(\lambda, A) > M$. So, by Lemma~\ref{LE:pbasics}, there exists a $\mu \in \R^m_+\setminus \{\mathbf{0}\}$ such that (i)  $\|\mu\|_1\leq \|\lambda\|_1-1$,  (ii) $(\mu A,\mu b)\in  \Pi_{\pp}$, and, (iii) $\mu A x\leq \floor{\mu b}_{S,\mu A}$ dominates $\lambda A x\leq \floor{\lambda b}_{S,\lambda A}$. As we argued in the proof of Theorem~\ref{thm:cov}, after repeating this process for at most $\|\lambda\|_1$ iterations, we may assume that $r(\mu, A) \leq M$ and  $(\mu A,\mu b)\in \Pi$. Since the \scg  generated by $\mu$ dominates the \scg  generated by $\lambda$ for $\pp$, it follows that $\pp_S=\pp_{S,\Pi}$. Since $\pp_{S,\Pi}$ is a rational polyhedron by Lemma~\ref{lemma:packMstar}, it follows that $\pp_S$ is a rational polyhedron, as required.
\ifx\flagJournal\true \qed \fi
\end{proof}

\subsection{General pointed polyhedra}\label{sec:proof}

\newcommand {\qc}{{Q^{\uparrow}}}
\newcommand {\qp}{{Q^{\downarrow}}}

By Theorems~\ref{thm:cov} and~\ref{thm:pack}, we now know that Theorem~\ref{quote} holds. Having proved Theorem~\ref{quote}, we are very close to finishing the proof of the main result of this section that when $S$ is the set of integer points in a rational pointed polyhedron, the \scgc of any rational polyhedron $P\subseteq\conv(S)$ is a rational polyhedron. 

Here comes a brief outline of our proof. We will first show that the following lemma holds, based on Theorem~\ref{quote}.

\begin{lemma}\label{lemma:double-star}
Let $T\subseteq \mathbb{Z}^n$ be such that $\conv(T)\cap\mathbb{Z}^n=T$ and $\conv(T)$ is of the form~\eqref{eq:S_pointed} satisfying~\eqref{eq:S_pointed_assumption}. Let $Q\subseteq \conv(T)$ be a rational polyhedron, and let $\Pi_Q^{+},\Pi_Q^{-}$ be defined as follows:
	\begin{align}\label{eq:foo}
	\begin{aligned}
	\Pi_Q^{+}&=\left\{(\alpha,\beta)\in\Pi_Q:
	\alpha y\geq 0,\ \text{for } y\in\left\{e^1,\ldots,e^{n_1},r^1,\ldots,r^h\right\}\right\},\\
	\Pi_Q^{-}& =\left\{(\alpha,\beta)\in\Pi_Q:
	\alpha y\leq 0,\ \text{for~} y\in\left\{e^1,\ldots,e^{n_1},r^1,\ldots,r^h\right\}
	\right\}.
	\end{aligned}
	\end{align}
	Then both $Q_{T,\Pi_Q^{+}}$ and $Q_{T,\Pi_Q^{-}}$ are rational polyhedra.
\end{lemma}
Before we prove this lemma, we highlight its connection to Theorem~\ref{quote}. If $T\subseteq\mathbb{Z}^n$ and $Q\subseteq \conv(T)$ satisfy the conditions of Lemma~\ref{lemma:double-star}, then $\qc:=Q+\mathcal{C}$ and $\qp:=Q-\mathcal{C}$ satisfy the conditions of Theorem~\ref{quote} (we will argue this formally later). Then, by Theorems~\ref{thm:cov} and~\ref{thm:pack}, $\qc_T$ and $\qp_T$ are rational polyhedra, based on which we will argue that $Q_{T,\Pi_Q^{+}}$ and $Q_{T,\Pi_Q^{-}}$ are rational polyhedra to complete the proof of Lemma~\ref{lemma:double-star}.

After proving Lemma~\ref{lemma:double-star}, the remaining part is basically to reduce the general setting to the case in Lemma~\ref{lemma:double-star}, thereby justifying that it is enough to consider the case in Lemma~\ref{lemma:double-star}. Note that $\Pi_Q^{+}$ and $\Pi_Q^{-}$ focus on inequalities $\alpha x\leq \beta$ where the signs of $\alpha e^1,\ldots, \alpha e^{n_1}$ and the signs of $\alpha r^1,\ldots, \alpha r^h$ are uniform. We will argue this in Lemmas~\ref{LE:partition} and~\ref{reduction}.

\begin{proof-lemma}{\rm
	Let $\qc$ and $\qp$ be defined as $\qc:=Q+\mathcal{C}$ and $\qp:=Q-\mathcal{C}$, respectively. By~\eqref{eq:S_pointed_assumption} and $Q\subseteq \conv(T)\subseteq \mathcal{C}$, it follows that $\qc$ is pointed and the extreme points of $\qc$ are contained in $\conv(T)$. Moreover, $\qc$ and $\qp$ can be written as $\qc=\left\{x\in\mathbb{R}^n:\;Ax\geq b\right\}$ and $\qp=\left\{x\in\mathbb{R}^n:\;Cx\leq d\right\}$
	where $A,b,C,d$ are matrices satisfying
	\begin{align}
	&Ax\geq\mathbf{0}\text{ for all }x\in \left\{e^1,\ldots,e^{n_1},r^1,\ldots,r^{h}\right\}\quad\text{and}\quad b\geq\mathbf{0},\label{cov-noneg}\\
	&Cx\geq\mathbf{0}\text{ for all }x\in \left\{e^1,\ldots,e^{n_1},r^1,\ldots,r^{h}\right\}\quad\text{and}\quad d\geq\mathbf{0}.\label{pack-noneg}
	\end{align}
	We first claim that $\qc_{T}\cap Q=Q_{T,\Pi_Q^-}$.We will show that $\Pi_Q^-=\Gamma$ where
	\[
	\Gamma=\left\{ (-\lambda A,-\lambda b) \in \Z^n\times\R:\;\lambda\in\mathbb{R}_+^m,\;{\lambda b}= \min\{\lambda A x : x \in \qc\}\right\}
	\]
	Let $(-\alpha,-\beta)\in \Gamma$. Then $\alpha x\geq \beta$ is a valid inequality for $\qc$. By~\eqref{cov-noneg}, it follows that $\alpha x\geq0$ for $x\in\left\{e^1,\ldots,e^{n_1},r^1,\ldots,r^{h}\right\}$, so $\min\{\alpha x : x \in \qc\}=\min\{\alpha x : x \in Q\}$. Then $-\beta=\max\{-\alpha x : x \in Q\}$, so $(-\alpha,-\beta)\in \Pi_Q^-$. Conversely, take $(-\alpha,-\beta)\in \Pi_Q^-$. Then $-\beta =\max\left\{-\alpha x:x\in Q\right\}$, so $\beta =\min\left\{\alpha x:x\in Q\right\}$. As $\alpha x\geq0$ for $x\in\left\{e^1,\ldots,e^{n_1},r^1,\ldots,r^{h}\right\}$, it follows that $\min\left\{\alpha x:x\in Q\right\}=\min\left\{\alpha x:x\in \qc\right\}$, and therefore, $(\alpha,\beta)=(\lambda A,\lambda b)$ for some $\lambda\in\mathbb{R}_+^m$ and $(-\alpha,-\beta)\in \Gamma$. Therefore, as $\Pi_Q^-=\Gamma$, we have $Q_{T,\Pi_Q^-}=\left\{x\in Q:\;\alpha x\geq \ceil{\beta}_{T,\alpha}\;\forall (-\alpha,-\beta)\in \Gamma\right\}=Q\cap \qc_{T}$, as required.
	
	Similarly, we claim that $\qp_{T}\cap Q=Q_{T,\Pi_Q^+}$. We will show that $\Pi_{\qp}=\Pi_Q^+$. Let $(\alpha,\beta)\in \Pi_{\qp}$.	Then $\alpha x\leq \beta$ is a valid inequality for $\qp$. By~\eqref{pack-noneg}, it follows that $\alpha x\geq0$ for $x\in\left\{e^1,\ldots,e^{n_1},r^1,\ldots,r^{h}\right\}$, which means that $\max\{\alpha x : x \in \qp\}=\max\{\alpha x : x \in Q\}$. So, it follows that $(\alpha,\beta)\in \Pi_Q^+$. Conversely, take $(\alpha,\beta)\in \Pi_Q^+$. Then, as $\alpha x\geq0$ for $x\in\left\{e^1,\ldots,e^{n_1},r^1,\ldots,r^{h}\right\}$ and $\beta=\max\{\alpha x : x \in Q\}$, it follows that $\beta=\max\{\alpha x : x \in \qp\}$. This implies that $(\alpha,\beta)\in \Pi_{\qp}$. Therefore, as $\Pi_{\qp}=\Pi_Q^+$, we obtain $Q_{T,\Pi_Q^+}=\left\{x\in Q:\;\alpha x\leq \floor{\beta}_{T,\alpha}\;\forall (\alpha,\beta)\in \Pi_{\qp}\right\}=P\cap \qp_{T}$, as required.		
	
	By Theorems~\ref{thm:cov} and~\ref{thm:pack}, both $\qc_{T}$ and $\qp_{T}$ are rational polyhedra. In turn, both $Q_{T,\Pi_Q^+}$ and $Q_{T,\Pi_Q^-}$ are rational polyhedra, as required.
	\ifx\flagJournal\true \qed \fi	
}
\end{proof-lemma}

Next, as we promised, we show that it is sufficient to consider inequalities $\alpha x\leq \beta$ where $\alpha y\geq0$ for all $y\in\left\{e^1,\ldots,e^{n_1},r^1,\ldots,r^h\right\}$ or $\alpha y\leq0$ for all $y\in\left\{e^1,\ldots,e^{n_1},r^1,\ldots,r^h\right\}$. By Lemma~\ref{reduction1}, we may focus on $S\subseteq \mathbb{Z}^n$ that has the property that $\conv(S)\cap\mathbb{Z}^n=S$ and $\conv(S)$ is of the form~\eqref{eq:S_pointed} satisfying~\eqref{eq:S_pointed_assumption}. Now we take a relaxation $S_0$ of $S$; we choose $S_0$ to be the set of integer points in $\conv\{v^1,\ldots,v^{g}\}+\lin\{r^1,\ldots,r^{h}\}$, which contains $\conv(S)$. 
\begin{figure}[h!]
	\begin{center}
		\begin{tikzpicture}
		[main node/.style={circle,fill=black!30,draw,minimum size=0.1em, inner sep=1pt}]
		
		\fill[blue, nearly transparent] (8,3) -- (7,2) -- (5,1) -- (5,2) -- (6,3)  -- (8,3);
		\draw[blue,thick] (8,3) -- (7,2) -- (5,1) -- (5,2) -- (6,3);
		
		\fill[blue, nearly transparent] (15,3) -- (13,1) -- (11,1) -- (13,3)  -- (15,3);
		\draw[blue,thick] (15,3) -- (13,1);
		\draw[blue,thick] (13,3) -- (11,1);
		
		\node[main node] (4) at (5,1) {};
		\node[main node] (5) at (4,1) {};
		\node[main node] (6) at (6,1) {};
		\node[main node] (7) at (7,1) {};
		\node[main node] (8) at (8,1) {};
		
		\node[main node] (12) at (4,2) {};
		\node[main node] (13) at (5,2) {};
		\node[main node,label=above right:{$S$}] (14) at (6,2) {};
		\node[main node] (15) at (7,2) {};
		\node[main node] (16) at (8,2) {};
		
		\node[main node] (20) at (4,3) {};
		\node[main node] (21) at (5,3) {};
		\node[main node] (22) at (6,3) {};
		\node[main node] (23) at (7,3) {};
		\node[main node] (24) at (8,3) {};

		\node[main node] (4) at (12,1) {};
		\node[main node] (5) at (11,1) {};
		\node[main node] (6) at (13,1) {};
		\node[main node] (7) at (14,1) {};
		\node[main node] (8) at (15,1) {};
		
		\node[main node] (12) at (11,2) {};
		\node[main node] (13) at (12,2) {};
		\node[main node,label=above right:{$S_0$}] (14) at (13,2) {};
		\node[main node] (15) at (14,2) {};
		\node[main node] (16) at (15,2) {};
		
		\node[main node] (20) at (11,3) {};
		\node[main node] (21) at (12,3) {};
		\node[main node] (22) at (13,3) {};
		\node[main node] (23) at (14,3) {};
		\node[main node] (24) at (15,3) {};
		
		\end{tikzpicture}
		\caption{$S$ and $S_0$}\label{fig:pointed-to-cylinder}
	\end{center}
\end{figure}
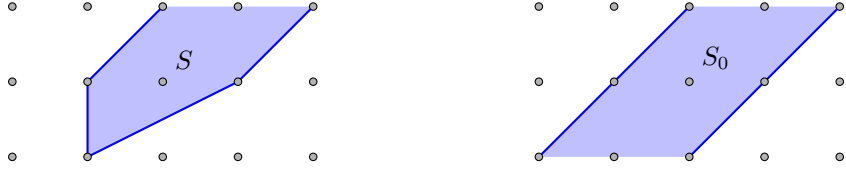
By definition, $\conv(S_0)\cap\mathbb{Z}^n=S_0$ and $S\subseteq S_0$. Moreover, since $\conv\{v^1,\ldots,v^{g}\}+\lin\{r^1,\ldots,r^{h}\}$ is an integral polyhedron and is also a rational cylinder,
\begin{equation}\label{eq:S_0}
\conv(S_0)=\conv\{v^1,\ldots,v^{g}\}+\lin\{r^1,\ldots,r^{h}\}
\end{equation}
and $\conv(S_0)$ is a rational cylinder (see Figure~\ref{fig:pointed-to-cylinder} for illustration). By Theorem~\ref{thm:cylinder}, we already know that $P_{S_0}$ is a rational polyhedron. Hence, we may focus on $S$-CG cuts that cut off some point in $P_{S_0}$. As a first step toward understanding such $S$-CG cuts, we observe the following lemma:

\begin{lemma}\label{LE:partition}
Let $S\subseteq \mathbb{Z}^n$ be such that $\conv(S)\cap\mathbb{Z}^n=S$ and $\conv(S)$ is of the form~\eqref{eq:S_pointed} satisfying~\eqref{eq:S_pointed_assumption}, and let $S_0\subseteq \mathbb{Z}^n$ be defined as above. Let $N_r=\{1,\ldots,h\}$. If $P\subseteq \conv(S)$ is a rational polyhedron, then $P_{S}=P_{S_0}\cap P_{S,\Pi^{+}}\cap P_{S,\Pi^{-}}$ where 
\begin{align}\label{eq:Pi_P+-}
\begin{aligned}
&\Pi^{+}:=\left\{(\alpha,\beta)\in\Pi_P:\;\alpha r^i\geq0~\text{ for }~i\in N_r\right\},\\
&\Pi^{-}:=\left\{(\alpha,\beta)\in\Pi_P:\;\alpha r^i\leq0~\text{ for }~i\in N_r\right\}.
\end{aligned}
\end{align}
\end{lemma}
\begin{proof}
As $S\subseteq S_0$, it follows that $P_{S}\subseteq P_{S_0}$. To prove the claim in the lemma, we will argue that if the \scg derived from $(\alpha,\beta)\in \Pi_P$ is violated by a point in $P_{S_0}$, then $(\alpha,\beta)\in \Pi^+\cup \Pi^-$. To this end, consider an arbitrary $(\alpha,\beta)\in \Pi_P$ such that $\alpha x \leq \floor{\beta}_{S,\alpha}$ is violated by a point in $P_{S_0}$. If $\floor{\beta}_{S,\alpha} =  \floor{\beta}_{S_0,\alpha}$, then the associated \scg is the same as the associated $S_0$-CG cut. Therefore $\floor{\beta}_{S,\alpha} < \floor{\beta}_{S_0,\alpha}$. This means that while $S_0$ contains a point $\bar z$ such that $\alpha \bar z=\floor{\beta}_{S_0,\alpha}$, there is no such point in $S$.
	
We will argue that either $\alpha r^i \geq 0$ for all $i\in N_r=\{1,\ldots, h\}$ or $\alpha r^i\leq0$ for all $i\in N_r$ must hold. Suppose for a contradiction that there are distinct $i,j\in   N_r$ such that $\alpha r^i>0$ and $\alpha r^j<0$. Let $J^+=\{i\in  N_r:\alpha r^i>0\}$ and $J^-:=\{j\in  N_r:\alpha r^j<0\}$. We construct a vector $r\in \mathbb{Z}^n$ where
	\[
	r:=\bigg(\sum_{i\in J^+}\alpha r^i\bigg)\sum_{j\in J^-}r^j+\bigg(-\sum_{j\in J^-}\alpha r^j\bigg)\sum_{i\in J^+}r^i.
	\]
	Since both $\sum_{i\in J^+}\alpha r^i$ and $-\sum_{j\in J^-}\alpha r^j$ are strictly positive, there exists an integer $M$ such that $\bar z+Mr\in S$. Moreover, note that $\alpha r=0$, and therefore, $\alpha(\bar z+Mr)=\alpha \bar z$. However, this implies that $\floor{\beta}_{S,\alpha}=\floor{\beta}_{S_0,\alpha}$, a contradiction. Therefore, $\alpha r^i\geq0$ for all $i\in N_r$ or $\alpha r^i\leq0$ for all $i\in N_r$ must hold.
	\ifx\flagJournal\true \qed \fi
\end{proof}

What we observed while proving Lemma~\ref{LE:partition} is that if $\alpha r^i>0>\alpha r^j$ for some $i,j\in N_r$, then the \scg from $\alpha x\leq \beta$ is not strictly stronger than the $S_0$-CG cut from $\alpha x\leq \beta$. Figure~\ref{fig:uniform-sign} provides a geometric intuition behind it. If $\alpha r^i>0>\alpha r^j$, then the hyperplane $\{x\in\R^n:\alpha x=\beta\}$ must contain a ray in the recession cone of $\conv(S)$, so the intersection of the hyperplane and $\conv(S)$ stretches toward the infinite direction. As a consequence, the hyperplane is surrounded by infinitely many integer points in $S$ that are potentially blocking the hyperplane from being moved by much without touching them.
\begin{figure}[h!]
	\begin{center}
		\begin{tikzpicture}
		[main node/.style={circle,fill=black!30,draw,minimum size=0.1em, inner sep=1pt}]
		
		\fill[blue, nearly transparent] (4.7,3) -- (4,1) -- (9,2.25) -- (9,3) -- (4.7,3);
		\draw[blue,thick] (4.7,3) -- (4,1) -- (9,2.25);
		\draw[red,thick] (4.2,3) -- (6.5,1);
		\draw[red,thick] (3.5,1.1) -- (7.5,3);

		\node[main node] (4) at (5,1) {};
		\node[main node] (5) at (4,1) {};
		\node[main node] (6) at (6,1) {};
		\node[main node] (7) at (7,1) {};
		\node[main node] (8) at (8,1) {};
		
		\node[main node] (12) at (4,2) {};
		\node[main node] (13) at (5,2) {};
		\node[main node] (14) at (6,2) {};
		\node[main node] (15) at (7,2) {};
		\node[main node] (16) at (8,2) {};
		
		\node[main node] (20) at (4,3) {};
		\node[main node] (21) at (5,3) {};
		\node[main node] (22) at (6,3) {};
		\node[main node] (23) at (7,3) {};
		\node[main node] (24) at (8,3) {};

		\node[main node] (5) at (9,3) {};
		\node[main node] (7) at (9,2) {};
		\node[main node] (8) at (9,1) {};

		\end{tikzpicture}
		\caption{Uniform signs versus nonuniform signs}\label{fig:uniform-sign}
	\end{center}
\end{figure}
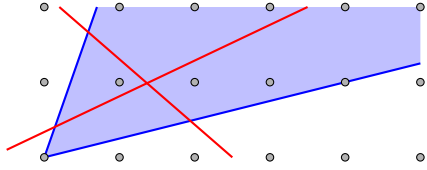
As a result, we may focus on inequalities $\alpha x\leq \beta$ where $\alpha y\geq0$ for all $y\in\left\{r^1,\ldots,r^h\right\}$ or $\alpha y\leq0$ for all $y\in\left\{r^1,\ldots,r^h\right\}$. 

However, in $\Pi_Q^{+}\cup\Pi_Q^{-}$, not only the signs of $\alpha r^1,\ldots, \alpha r^h$ are the same but also the signs of $\alpha e^1,\ldots, \alpha e^{n_1}$. As $\rec(\conv(S))\subseteq \{\mathbf{0}\}\times\mathbb{R}^{n_2}$ by~\eqref{eq:S_pointed_assumption}, $e^1,\ldots, e^{n_1}$ correspond to the coordinates in which the points of $S$ have bounded values. Hence, the intuition about having an infinite intersection does not hold here, so we need a separate technique. That is, we partition $\Pi^+$ and $\Pi^-$ in~\eqref{eq:Pi_P+-} based on the sign pattern of $\alpha e^1,\ldots, \alpha e^{n_1}$ and for each part of the partition, apply a unimodular transformation to make $\alpha e^1,\ldots,\alpha e^n$ have the same sign. 

\begin{lemma}\label{reduction}
Let $S\subseteq \mathbb{Z}^n$ be such that $\conv(S)\cap\mathbb{Z}^n=S$ and $\conv(S)$ is of the form~\eqref{eq:S_pointed} satisfying~\eqref{eq:S_pointed_assumption}, and let $P\subseteq \conv(S)$ be a rational polyhedron. Then $P_{S}$ is a rational polyhedron.	
\end{lemma}
\begin{proof}
By Lemma~\ref{LE:partition}, to show that $P_{S}$ is a rational polyhedron, it is sufficient to show that both $P_{S,\Pi^{+}}$ and $P_{S,\Pi^{-}}$ are rational polyhedra, where $\Pi^{+}$ and $\Pi^{-}$ are defined as in~\eqref{eq:Pi_P+-}. By~\eqref{eq:S_pointed_assumption}, we have $\lin\left(\conv(S)\right)=\{\mathbf{0}\}\times \mathbb{R}^{n_2}$ and thus $\left(\lin\left(\conv(S)\right)\right)^\perp=\mathbb{R}^{n_1}\times\{\mathbf{0}\}$. So, $\left\{e^1,\ldots,e^{n_1}\right\}$ is a basis of $\left(\lin\left(\conv(S)\right)\right)^\perp$. Next we partition $\Pi^{+}$ and $\Pi^{-}$, according to the sign pattern of $\pi e^1,\ldots,\pi e^{n_1}$. Let $N_1=\{1,\ldots,n_1\}$, and for $J\subseteq N_1$, we let
\begin{align*}
	&\Pi^{+}(J)=\left\{(\pi,\beta)\in\Pi^{+}:\;
	\pi e^j\geq0~\text{ for }~j\in J,\;\pi e^j\leq0~\text{ for }~j\in  N_1\setminus J
	\right\},\\
	&\Pi^{-}(J) =\left\{(\pi,\beta)\in\Pi^{-}:\;
	\pi e^j\leq0~\text{ for }~j\in J,\;\pi e^j\geq0~\text{ for }~j\in  N_1\setminus J
	\right\}.
	\end{align*}
	Then it follows from Lemma~\ref{LE:partition} that 
	\begin{equation}	
	P_{S}=P_{S_0}\cap \big(\cap_{J\subseteq   N_1}P_{S,\Pi^{+}(J)} \big)\cap \big(\cap_{J\subseteq   N_1}P_{S,\Pi^{-}(J)}\big).\label{eq:decompose}
	\end{equation}
In fact, based on Lemmas~\ref{lemma:double-star} and~\ref{LE:unimodular-closure}, we will argue that $P_{S,\Pi^{+}(J)}$  and  $P_{S,\Pi^{+\-}(J)}$ are rational polyhedra for all $J\subseteq  N_1$. To this end, take a $J\subseteq   N_1$, and let $\tau$ be the unimodular transformation mapping $x\in\mathbb{R}^n$ to $y=\tau(x)\in\mathbb{R}^n$ where
	\[
	y_i:=\begin{cases}
	\begin{array}{ll}
	-x_i,&\text{if }i\in   N_1\setminus J\\
	x_i,&\text{otherwise.}
	\end{array}
	\end{cases}
	\]
	Let $Q:=\tau(P)$ and $T:=\tau(S)$. Clearly, $T\subseteq \mathbb{Z}^n$ satisfies that $\conv(T)\cap\mathbb{Z}^n=T$ and $\conv(T)$ is of the form~\eqref{eq:S_pointed} satisfying~\eqref{eq:S_pointed_assumption}. Moreover, it follows from Lemma~\ref{LE:unimodular-closure} that $Q\subseteq \conv(T)$, $\tau(P_{S,\Pi^{+}(J)})=Q_{T,\Pi_Q^+}$, and $\tau(P_{S,\Pi^{-}(J)})=Q_{T,\Pi_Q^{-}}$ where $\Pi_Q^+$ and $\Pi_Q^-$ are defined as in~\eqref{eq:foo}, respectively. Then, by Lemma~\ref{lemma:double-star}, $Q_{T,\Pi_Q^{+}}$ and $Q_{T,\Pi_Q^{-}}$ are rational polyhedra, implying in turn that $P_{S,\Pi^{+}(J)}$ and $P_{S,\Pi^{-}(J)}$ are rational polyhedra. So, by~\eqref{eq:decompose}, $P_{S}$ is a rational polyhedron.
	\ifx\flagJournal\true \qed \fi
\end{proof}

Now we are finally ready to prove the main result of this section.

\begin{theorem}\label{thm:pointed}
Let $S=R\cap\mathbb{Z}^n$ for some rational pointed polyhedron $R$, and let $P\subseteq \conv(S)$ be a rational polyhedron. Then $P_S$ is a rational polyhedron.
\end{theorem}
\begin{proof}
By the unimodular mapping lemma (Lemma~\ref{LE:unimodular-closure}) and Lemma~\ref{reduction1}, we may assume that $\conv(S)\cap\mathbb{Z}^n=S$ and $\conv(S)$ is of the form~\eqref{eq:S_pointed} satisfying~\eqref{eq:S_pointed_assumption}. Then, by Lemma~\ref{reduction}, $P_S$ is a rational polyhedron, as required.
	\ifx\flagJournal\true \qed \fi
\end{proof}

\section{Polyhedra with nontrivial lineality space}\label{sec:main}

In this section, we get back to the most general case and prove Theorem~\ref{main result}:
\[
S=R\cap\mathbb{Z}^n\quad \text{where }R\text{ is a rational polyhedron}
\]
and $R$ is not necessarily pointed. Then $\conv(S)\cap\mathbb{Z}^n=S$ and $\conv(S)$ can be written as
\begin{equation}\label{S:general}
\conv(S)=\mathcal{P}+\mathcal{R}+\mathcal{L}\nonumber
\end{equation}
where $\mathcal{L}$ is the lineality space of $\conv(S)$, $\mathcal{P}+\mathcal{R}$ is the pointed polyhedron $\conv(S)\cap\mathcal{L}^\perp$ whose recession cone is $\mathcal{R}$, and $\mathcal{P}$ is a polytope. As in the previous section, we take a relaxation $S_0\subseteq \mathbb{Z}^n$ such that $\conv(S_0)\cap\mathbb{Z}^n=S_0$ and $$\conv(S_0)=\mathcal{P}+\lin(\mathcal{R})+\mathcal{L}$$
where $\lin(\mathcal{R})$ is the linear hull of $\mathcal{R}$ or $\mathcal{R}+(-\mathcal{R})$. By definition, $S\subseteq S_0$ and $\conv(S_0)$ is a relaxation of $\conv(S)$. Moreover, $\conv(S_0)$ is a rational cylinder, and by Theorem~\ref{thm:cylinder}, we know that the $S_0$-CG closure of a rational polyhedron is a rational polyhedron.

\begin{lemma}\label{LE:general-2}
	If $P\subseteq \conv(S)$ is a rational polyhedron, then  
	\begin{equation}\label{eq:Pi}
	P_{S}=P_{S_0}\cap P_{S,\Pi}\quad\text{where}~~\Pi:=\left\{(\alpha,\beta)\in\Pi_P:\; \alpha \ell=0~\text{ for }~\ell\in \mathcal{L}\right\}.
	\end{equation}
\end{lemma}
\begin{proof}
As $S\subseteq S_0$, we know that $P_{S}\subseteq P_{S_0}$. We will argue that if the \scg derived from $(\alpha,\beta)\in \Pi_P$ cuts off a point in $P_{S_0}$, then $(\alpha,\beta)\in \Pi$, thereby proving that $P_{S}=P_{S_0}\cap P_{S,\Pi}$. To this end, take a pair $(\alpha,\beta)\in\Pi_P$. We may assume that $\floor{\beta}_{S,\alpha}<\floor{\beta}_{S_0,\alpha}$. Otherwise, $\floor{\beta}_{S,\alpha}=\floor{\beta}_{S_0,\alpha}$ and the \scg derived from $(\alpha,\beta)$ is the same as the corresponding $S_0$-CG cut, which means that the \scg does not cut off any point in $P_{S_0}$. Let $z\in S_0$ be such that $\alpha z = \floor{\beta}_{S_0,\alpha}$. Then the assumption $\floor{\beta}_{S,\alpha}<\floor{\beta}_{S_0,\alpha}$ implies that $\floor{\beta}_{S,\alpha}<\alpha z$.
	
Let $\left\{\ell^1,\ldots,\ell^{g}\right\}$ be a basis of $\mathcal{L}$, and let $\left\{r^1,\ldots,r^{h}\right\}$ be a basis of $\mathcal{R}$. It suffices to show that $\alpha\ell^i=0$ for all $i=1,\ldots,g$. Suppose for a contradiction that $\alpha \ell^{i}\neq 0$ for some $i$. Now we construct a vector $r$ as follows:
\[
r:=\left|\alpha \ell^{i}\right|\sum_{j=1}^hr^j - \frac{\left|\alpha e^{i}\right|}{\alpha \ell^{i}}\bigg(\sum_{j=1}^h\alpha r^j\bigg)\ell^{i}.
\]
Since $\left|\alpha \ell^{i}\right|$ is strictly positive, there exists a sufficiently large integer $M$ such that $z+Mr\in S$. Moreover, notice that $\alpha r=0$, so it follows that $\alpha(z+Nr)=\alpha z$. This in turn implies $\alpha z\leq \floor{\beta}_{S,\alpha}<\floor{\beta}_{S_0,\alpha}=\alpha z$, a contradiction. Therefore, $\alpha \ell^i=0$ for all $i$, as required. 
\ifx\flagJournal\true \qed \fi
\end{proof}
By this lemma, it is sufficient to show that $P_{S,\Pi}$ is a rational polyhedron, for which the following lemma will be useful:
\begin{lemma}[\bf Projection lemma~\cite{Dash19}]\label{LE:projection}
Let $F, S$ and $P$ be defined as 
\begin{equation}
S=F\times \mathbb{Z}^{n_2} \mbox{ for some } F \subseteq \mathbb{Z}^{n_1},\quad P = \{(x,y) \in \R^{n_1}\times\R^{n_2} : Ax + Cy \leq b\}\notag 
\end{equation}
where the matrices $A,C,b$ have integral components and $n_1,n_2,1$ columns, respectively. Let $\Omega \subseteq \{(\alpha, \beta) \in \Pi_P: \alpha=(\phi,\mathbf{0})\in \R^{n_1}\times\R^{n_2}\}$, and let $\Phi = \left\{(\phi,\beta)\in \R^{n_1}\times\R:(\phi,\mathbf{0})=\alpha,\;(\alpha,\beta)\in \Omega\right\}$. If $Q = \proj_x(P)$, then, $P_{S,\Omega}=P\cap\left(Q_{F, \Phi}\times\R^{n_2}\right)$.
\end{lemma}
Now we are ready to prove the main result of this paper:
\begin{proof-main theorem}{\rm 
Let $\Pi$ be defined as in~\eqref{eq:Pi}. By Lemma~\ref{LE:general-2}, we know that $P_{S}=P_{S_0}\cap P_{S,\Pi}$. Since $\conv(S_0)$ is a rational cylinder, Theorem~\ref{thm:cylinder} implies that $P_{S_0}$ is a rational polyhedron. So, it is sufficient to show that $P_{S,\Pi}$ is a rational polyhedron. Since $\mathcal{P}+\mathcal{R}=\conv(S)\cap\mathcal{L}^\perp$, there exists a unimodular transformation $\tau$ such that $\tau(\mathcal{L})=\left\{\bf 0\right\}\times\mathbb{R}^{n_2}$ and $\tau(\mathcal{P}+\mathcal{R})\subseteq\mathbb{R}^{n_1}\times\left\{\bf 0\right\}$. Let $Q=\tau(P)$ and $T=\tau(S)$. By Lemma~\ref{LE:unimodular-closure},
\[
\tau(P_{S,\Pi})=Q_{T,\Omega}\quad\text{where}~~\Omega:=\{(\alpha,\beta)\in \Pi_Q:\;\alpha\ell=0~\text{ for }~\ell\in \tau(\mathcal{L})\}.
\]
As $\tau(\mathcal{L})=\left\{\bf 0\right\}\times\mathbb{R}^{n_2}$, we have $\Omega=\left\{(\alpha,\beta)\in\Pi_Q: \alpha_{n_1+1}=\cdots=\alpha_{n_1+n_2}=0\right\}$. Moreover, $T$ can be written as $T=T_C\times \mathbb{Z}^{n_2}$ where $\conv(T_C)\subseteq\mathbb{R}^{n_1}$ is a pointed polyhedron and $T_C=\conv(T_C)\cap\mathbb{Z}^{n_1}$. Let $$\Phi=\left\{(\phi,\beta)\in\mathbb{R}^{n_1}\times\mathbb{R}:\;(\phi,\mathbf{0})=\alpha,~(\alpha,\beta)\in \Omega\right\}.$$ Let $\hat Q$ denote the projection of $Q$ onto the $\mathbb{R}^{n_1}$-space. As $Q\subseteq \conv(T)$, we have $\hat Q\subseteq \conv(T_C)$ and $\Phi=\Pi_{\hat Q}$. Since $\conv(T_C)$ is pointed, we know from Theorem~\ref{thm:pointed} that $\hat Q_{T_C,\Phi}$ is a rational polyhedron. Since $Q_{T,\Omega}= Q\cap(\hat Q_{T_C,\Phi}\times\mathbb{R}^{n_2})$ by Lemma~\ref{LE:projection}, it follows that $Q_{T,\Omega}$ is a rational polyhedron, which implies that $P_{S,\Pi}$ is a rational polyhedron. Therefore, $P_S$ is a rational polyhedron, as required.
	\ifx\flagJournal\true \qed \fi}
\end{proof-main theorem}
	
\section{Mixed-integer setting}\label{sec:mixed}

The last setting of this paper is the mixed-integer case. Namely, $S\subseteq \Z^n\times\R^l$ is a mixed-integer set given by
\begin{equation}\label{eq:Smixed}
S=\left\{(x,y)\in \Z^n\times\R^l:\ Ax + Cy \leq b\right\}
\end{equation}
where $A,C,b$ are matrices of appropriate dimension with integer entries. Although $S$ is not a pure-integer set as before, the definition of \scgs and \scgc can be extended to this mixed-integer setting. Let $P\subseteq \R^n\times\R^l$ be a rational polyhedron. Now take a valid inequality for $P$ that has the form $\alpha x\leq \beta$ with $\alpha\in\Z^n$. Note that $\alpha x\leq \beta$ involves integer variables only and none of the continuous variables. With a slight abuse of notation, we define $\floor{\beta}_{S,\alpha}$ for the mixed-integer set $S$ as
\begin{equation}\label{eq:floor-mixed}
\floor{\beta}_{S,\alpha} := \max\{\alpha x : (x,y) \in S, \alpha x \leq \beta\}= \max\{\alpha x : x \in \proj_x(S), \alpha x \leq \beta\}.
\end{equation}
Then we can define the \scg derived from $\alpha x\leq \beta$ simply as $\alpha x\leq \floor{\beta}_{S,\alpha}$. Notice that $S$-CG cuts for a mixed-integer set $S$ are a generalization of projected Chv\'atal-Gomory cuts introduced by Bonami et al.~\cite{projected-cg}. Now that the \scgs are defined for the mixed-integer setting, we may define the $S$-CG closure accordingly. Basically, we obtain the \scgc of $P$ by applying all possible \scgs for $P$ as we did for the pure-integer case. A difference, however, is that inequalities defining $P$ in the mixed-integer case are not necessarily of the form $\alpha x\leq \beta$ that has no continuous variable. Hence, unlike in~\eqref{eq:scg} for the pure-integer case, we define the \scgc of $P$ as
\begin{equation}\label{eq:mixed}
P_S=\bigcap\limits_{\alpha\in \Z^n}\left\{(x,y)\in P:\;\alpha x\leq \floor{\max\{\alpha x :(x,y)\in P\}}_{S,\alpha}\right\}.
\end{equation}
Here, we take points from $P$ instead of the ambient space $\R^n\times\R^l$. By defining $P_S$ this way, we make sure that $P_S$ is contained in $P$. Note that when $l=0$ and thus $S$ is a pure-integer set, the definition~\eqref{eq:mixed} is consistent with~\eqref{eq:scg}. It turns out that we still have the polyhedrality result of \scgc under this definition.
\begin{theorem}\label{thm:mixed}
Let $S$ be a mixed-integer set given by~\eqref{eq:Smixed} and $P\subseteq \conv(S)$ be a rational polyhedron. Then the $S$-CG closure is a rational polyhedron.
\end{theorem}
\begin{proof}
By definition, every \scg involves no continuous variable. The idea is basically to project out the continuous variables and then use Theorem~\ref{main result}. We claim that for each $\alpha\in\Z^n$,
\begin{equation}\label{eq:claim}
\floor{\max\{\alpha x :\ (x,y)\in P\}}_{S,\alpha}= \floor{\max\{\alpha x :\ x\in \proj_x(P)\}}_{\proj_x(S),\alpha}
\end{equation}
where $\floor{\cdot}_{\proj_x(S),\alpha}$ is the floor function defined for the pure-integer case as $\proj_x(S)\subseteq\Z^n$. In fact, by~\eqref{eq:floor-mixed}, $\floor{\max\{\alpha x :\ (x,y)\in P\}}_{S,\alpha}=\floor{\max\{\alpha x :\ (x,y)\in P\}}_{\proj_x(S),\alpha}$
and~\eqref{eq:claim} follows because $\max\{\alpha x : (x,y)\in P\}=\max\{\alpha x : x\in \proj_x(P)\}$. For ease of notation, let $Q$ and $T$ denote $\proj_x(P)$ and $\proj_x(S)$, respectively. Then there exists a rational polyhedron $R$ such that $T= R\cap\Z^n$. Moreover, $Q\subseteq \conv(T)$ as $P\subseteq \conv(S)$ and $Q$ is a rational polyhedron. Now we are ready to complete the proof. Note that
\begin{align}
P_S&= P\cap \bigcap\limits_{\alpha\in \Z^n}\left\{(x,y)\in \R^n\times \R^l:\;\alpha x\leq \floor{\max\{\alpha x : (x,y)\in P\}}_{S,\alpha}\right\}\notag\\
&=P\cap \bigcap\limits_{\alpha\in \Z^n}\left\{(x,y)\in \R^n\times \R^l:\;\alpha x\leq \floor{\max\{\alpha x : x\in Q\}}_{T,\alpha}\right\}\label{5-1}\\
&=P\cap \bigg(\bigcap\limits_{\alpha\in \Z^n}\left\{x\in \R^n:\;\alpha x\leq \floor{\max\{\alpha x : x\in Q\}}_{T,\alpha}\right\}\times \R^l\bigg)\label{5-2}\\
&=P\cap \left(Q_T\times \R^l\right)\label{5-3}
\end{align}
where~\eqref{5-1} is from~\eqref{eq:claim},~\eqref{5-2} follows because none of the variables $y$ appears in the inequalities applied, and~\eqref{5-3} is from the definition of the $T$-CG closure of $Q$. By Theorem~\ref{main result}, we know that $Q_T$ can be described by a finitely many inequalities, implying in turn that $P_S$ is a rational polyhedron.
\ifx\flagJournal\true \qed \fi
\end{proof}	
\section{Conclusion}\label{sec:concl}

The main result in this paper, namely Theorem~\ref{main result}, is very general as the polyhedrality results of Schrijver~\cite{S1980}, Dunkel and Schulz \cite{DS2012}, and our previous paper \cite{Dash19} all follow from Theorem~\ref{main result}. In the same way that split cuts, lattice-free cuts~\cite{int4}, and $t$-branch split cuts \cite{int13,max-facet-width} generalize CG cuts, one can generalize $S$-CG cuts using $S$-splits (suggested in \cite{Dash19}), $S$-free convex sets \cite{int11}, and unions of $S$-splits, respectively. We note that intersection cuts from $S$-free convex sets were studied in \cite{int11} for infinite group relaxations of MIPs, where $S$ is defined as in our paper. It would be interesting to see if polyhedrality results for $S$-CG cuts can be extended to such families of cutting planes. Furthermore, it was observed in \cite{Dash19} that even testing the validity of an $S$-CG cut for the relatively simple case $S = \Z^n_+$ is already NP-hard. On the other hand, if $S$ is chosen to be a mixed-integer set with a bounded number of integral components, then given an inequality $\alpha x \leq \beta$, computing $\floor{\beta}_{S,\alpha}$ defined as in~\eqref{eq:floor-mixed} amounts to solving a mixed-integer program with a constant number of integer variables. In that case, $\floor{\beta}_{S,\alpha}$ can be computed in polynomial time in the encoding size of $S$, $\alpha$, and $\beta$.	

\paragraph{Acknowledgements}

We would like to thank two anonymous referees for their valuable feedback on this paper, and we also thank two anonymous referees for the IPCO version~\cite{Dash20} of this paper. This research was supported, in part, by the Institute for Basic Science (IBS-R029-C1, IBS-R029-Y2).

\bibliographystyle{abbrvnat}
\bibliography{scg}

\end{document}